\documentclass[a4paper,11pt,twoside]{amsart}
\usepackage[latin1]{inputenc}
\usepackage[T1]{fontenc}

\usepackage{a4wide}

\usepackage{ifpdf}
\ifpdf\usepackage[pdftex]{hyperref}
\else\usepackage[hypertex]{hyperref}\fi

\usepackage{amssymb}
\usepackage[british]{babel}

\theoremstyle{plain}
\newtheorem{thm}{Theorem}[section]
\newtheorem{prop}[thm]{Proposition}
\newtheorem{lemma}[thm]{Lemma}
\newtheorem{cor}[thm]{Corollary}

\theoremstyle{definition}
\newtheorem{defn}[thm]{Definition}

\theoremstyle{remark}
\newtheorem{rem}[thm]{Remark}

\theoremstyle{remark}

\DeclareMathOperator{\tr}{Tr}
\DeclareMathOperator{\rtr}{^\textrm{R}Tr}
\DeclareMathOperator{\tra}{tr}
\DeclareMathOperator{\id}{Id}
\DeclareMathOperator{\spec}{spec}
\DeclareMathOperator{\End}{End}

\DeclareMathOperator{\dvol}{d vol}
\DeclareMathOperator{\Diff}{Diff}
\DeclareMathOperator{\ind}{ind}

\renewcommand{\Re}{\mathop{\textnormal{Re}}}
\renewcommand{\Im}{\mathop{\textnormal{Im}}}

\let\dsp=\displaystyle 
\let\lra=\longrightarrow

\def\R{\mathbb R}
\def\C{\mathbb C}
\def\N{\mathbb N}
\def\Z{\mathbb Z}

\def\EE{\mathcal{E}}
\def\LL{\mathcal{L}}
\def\PP{\mathcal{P}}
\def\VV{\mathcal{V}}

\def\hook{\lrcorner\:}
\def\that{\widehat{\theta}}

\begin{document}

\title{Analytic torsions on contact manifolds}

\author{Michel Rumin}
\address{Laboratoire de Mathématiques d'Orsay\\
  CNRS et Université Paris Sud\\
  91405 Orsay Cedex\\ France}

\email{michel.rumin@math.u-psud.fr}

\author{Neil Seshadri}
\address{Graduate School of Mathematical Sciences\\
  The University of Tokyo\\
  3--8--1 Komaba, Meguro, Tokyo 150-0044\\ Japan}
\curraddr{Rates Hybrids Quantitative Research\\
  JPMorgan Securities Japan Co.~Ltd.\\
  Tokyo Building, 2-7-3 Marunouchi, Chiyoda-ku, Tokyo 100-6432\\
  Japan}

\email{neil.seshadri@hotmail.com}

\date{\today}

\begin{abstract}
  We propose a definition for analytic torsion of the contact complex
  on contact manifolds. We show it coincides with Ray--Singer torsion
  on any $3$-dimensional CR Seifert manifold equipped with a unitary
  representation.  In this particular case we compute it and relate it
  to dynamical properties of the Reeb flow. In fact the whole spectral
  torsion function we consider may be interpreted on CR Seifert
  manifolds as a purely dynamical function through Selberg-type trace
  formulae.
\end{abstract}



\thanks{First author supported in part by the French
  ANR-06-BLAN60154-01 grant. The second author gratefully acknowledges
  financial support received through the Japanese Government (MEXT)
  Scholarship for research students.}

\maketitle


\section{Introduction}
\label{sec:introduction}

As introduced by Ray and Singer in \cite{RS}, the analytic torsion of
a compact Riemannian manifold $M$ may be seen as an
infinite-dimensional analogue, on the de Rham complex $(\Omega^*M,
d)$, of the Reidemeister--Franz torsion of a finite simplicial
complex. More precisely, for $\lambda \geq 0$, let $E_\lambda^k $ be
the $]0, \lambda]$-spectral space of Hodge--de Rham Laplacian
$\Delta_k$ on $k$-forms.  Then the cut-off subcomplex $(E_\lambda^*,
d)$ is finite-dimensional and its Reidemeister--Franz torsion
satisfies
\begin{equation}
  \label{eq:1}
  2 \ln \tau_R (E_\lambda^*, d) = \ln \Bigl( \prod_{k=0}^n \det
  ({\Delta_k}_{\lceil E_\lambda^k} )^{(-1)^{k+1} k} \Bigr) =
  \sum_{k=0}^n (-1)^k k 
  \zeta'({\Delta_k}_{\lceil E_\lambda^k})(0) \,,
\end{equation}
where $\dsp \zeta({\Delta_k}_{\lceil E_\lambda^k})(s) = \tr\bigl(
{\Delta_k}_{\lceil E_\lambda^k}^{-s} \bigr)$ is the truncated zeta
function of $\Delta_k$ on $E^k_\lambda$. Taking these $(E_\lambda^*,
d)$ as successive `approximations' to the full de Rham complex,
Ray--Singer defined the analytic torsion $T_{RS}$ as being
\begin{equation}
  \label{eq:2}
  T_{RS} = \exp \Bigl( \frac{1}{2} \sum_{k=0}^n (-1)^k k \zeta'(\Delta_k)(0)
  \Bigr)\,,
\end{equation}
while the Ray--Singer metric on $\mathcal{L} = \det H^*(\Omega^*M, d)$
is given by
\begin{equation}
  \label{eq:3}
  \| \quad \|_{RS} = (T_{RS})^{-1} \mid \quad |_{L^2(\Omega^* M)},
\end{equation}
from the $L^2$ metric induced on $\mathcal{L}$ via identification of
the cohomology by harmonic forms in $\Omega^* M$.

The first purpose of this work is to adapt this idea to the contact
complex $(\mathcal{E}^*, d_H)$, a hypoelliptic differential form
complex naturally defined (\cite{Rumin90, Rumin94}) on contact
manifolds $(M ,H)$ of dimension $2n+1$. A specific feature of this
complex is that the differential $D = d_H~:~\mathcal{E}^n \rightarrow
\mathcal{E}^{n+1}$ in `middle degree' is a second-order operator,
which is due to a slower spectral sequence convergence at this degree;
see \cite[Proposition 3.3]{Rumin00}. In order to find, as above,
finite-dimensional cut-off subcomplexes $(E_\lambda^k, d_H)$
approximating $(\mathcal{E}^*, d_H)$, we are led to consider
\emph{fourth}-order Laplacians $\Delta_k$ in \emph{all} degrees $k$;
see \eqref{eq:10}. The Reidemeister--Franz torsion of each cut-off
subcomplex is then easily written (see
Proposition~\ref{prop:Reidemeister-contact}) as
\begin{equation*}
  4 \ln \tau_R (E_\lambda^*, d_H) = \ln \Bigl( \prod_{k=0}^{2n+1} \det
  ({\Delta_k}_{\lceil E_\lambda^k})^{(-1)^{k+1} w(k)} \Bigr) =
  \sum_{k=0}^{2n+1} 
  (-1)^k w(k) \zeta'({\Delta_k}_{\lceil E_\lambda^k})(0) \,,
\end{equation*}
with $w(k) = k $ for $k\leq n$ and $w(k)= k+1$ for $k> n$ being the
natural contact-weight of forms in $\mathcal{E}^k$; compare with
\eqref{eq:1}. This leads us to define a candidate for the analytic
torsion of the full contact complex by setting
\begin{equation*}
  T_C = \exp \Bigl(\frac{1}{4} \sum_{k=0}^{2n+1} (-1)^{k+1} w(k)
  \zeta'(\Delta_k)(0) \Bigr) \,.
\end{equation*}
We define also a torsion function
\begin{equation*}
  \kappa(s)= \frac{1}{2} \sum_{k=0}^{2n+1} (-1)^{k+1} w(k)
  \zeta(\Delta_k)(s) \,,
\end{equation*}
and a Ray--Singer metric $\|\quad \|_C$ on the determinant of the
cohomology $\det H^*(\mathcal{E}, d_H)$,
\begin{equation*}
  \|\quad \|_C = T_C \mid \quad |_{L^2(\mathcal{E})}\,.
\end{equation*}
(Our convention for $T_C$ is inverse to Ray--Singer's original
definition \eqref{eq:2}, but is standard now since it is natural at
the metric level, as compared to \eqref{eq:3}; see
e.g.~\cite{BZ,BGSIII}.)

\smallskip

Having thus defined a torsion upon geometric and algebraic bases, we
start then its analytical study.  We first establish in
Theorem~\ref{thm:contact-tor-var} and
Corollary~\ref{cor:contact-quillen-var} general variational formulae
for $\kappa(0)$, $T_C= \exp(\kappa'(0)/2)$ and the Ray--Singer
`contact' metric $\|\quad \|_C$. It turns out that $\kappa(0)$ is a
\emph{contact} invariant given by the integration of an unknown
universal polynomial in local curvature data. Its vanishing is
necessary in order for the Ray--Singer metric to be even scale
invariant under change of contact form $\theta\mapsto K \theta$ for
$K$ constant. We do not know whether $\kappa(0)$ vanishes in general
except in dimension $3$, as shown in Corollary~\ref{cor:3var}.
Therefore the rest of the paper deals with this lowest-dimensional
case.

Corollary~\ref{cor:3var} also states that there exist CR-invariant and
contact-invariant `corrections' to this Ray--Singer metric. Namely
there exist universal constants $(C_i)_{1\leq i\leq 4}$ such that, on
any contact manifold of dimension $3$,
$$
\|\quad \|_{\mathrm{CR}} = \exp\left( C_1 \int_M R^2 \,
  \theta\wedge d\theta + C_2 \int_M |A|^2 \theta\wedge d\theta\right)
\|\quad \|_{C} \,,
$$
is a CR-invariant (i.e.~independent of contact form) metric on
$\det H^*(\mathcal{E},d_H)$, where $R$ and $A$ are the Tanaka--Webster
scalar curvature and torsion.  Moreover
$$
\|\quad \|_{H}^\nu = \exp(C_3 \nu(M)) \|\quad \|_{\mathrm{CR}}
\quad \mathrm{and} \quad \|\quad \|_{H}^{D} = \exp(C_4
\overline{\eta}(D*)) \|\quad \|_{\mathrm{CR}}
$$
are contact-invariant metrics, where $\nu(M)$ is the
$\nu$-invariant of Biquard--Herzlich~\cite{BH}, and
$\overline{\eta}(D*)$ is the CR-invariant correction to $\eta(D*)$;
see \cite[Theorem 9.4]{BHR}.

We do not know the values of the constants $C_i$. Though they are all
related to the Heisenberg symbol of the Laplacians we consider, they
might be difficult to compute: first since the $\Delta_k$ are
fourth-order, but also because the Heisenberg symbolic calculus
(\cite{BGS,Getzler}, \S\ref{subsec:heisenberg}) suitable for the
hypoelliptic contact complex, is highly non-commutative.

\smallskip

Our next purpose is to compare the two analytic torsions and metrics
coming from the de Rham and contact complexes. It is natural to expect
they are related. Indeed, firstly these complexes have the same
cohomology, being homotopy equivalent
(Proposition~\ref{prop:contact-deRham}), and in particular the
determinants $\det H^*(\Omega^* M, d)$ and $\det H^*(\mathcal{E},
d_H)$ are canonically isomorphic.  Moreover, from the point of view of
spectral geometry, the non-exploding part of the Hodge--de Rham
spectrum converges towards the contact complex spectrum when one takes
the sub-Riemannian (or diabatic) limit $\varepsilon \searrow 0$ of
calibrated metrics $g_\varepsilon = d \theta(\cdot, J \cdot) +
\varepsilon^{-1} \theta^2$; see \cite{Rumin00, BHR}. Note that the
classical Ray--Singer metric stays constant in this limit, being
independent of the metric on $M$.

However, we cannot prove equality of metrics in general, but only on
particular contact manifolds called \emph{CR Seifert manifolds} in
\cite{BHR}.  These are CR manifolds $(M,H, J)$ of dimension $3$
admitting a transverse locally free circle action preserving the CR
structure $(H, J)$; see Definition~\ref{def:CR-Seifert}. The generator
$T = d/d t$ of the circle action is the Reeb field of an invariant
contact form $\theta$. On such a manifold $M$, endowed with any
unitary representation $\rho:\pi_1(M) \rightarrow U(N)$,
Theorem~\ref{thm:RS} states that the two Ray--Singer metrics of the
twisted de Rham and contact complexes coincide.

\smallskip

In the last part of this work we analyse in detail the torsion
function $\kappa(s)$ for CR Seifert manifolds. It first turns out that
$\kappa(s)$ is a dynamical function in this case, depending only on
the topology of $M$, together with the holonomies of the
representation along the various primitive closed orbits of the circle
action, as stated in formula \eqref{eq:40} and
Theorem~\ref{thm:kappa_x}.

Specialising to $s=0$ leads in Theorem~\ref{thm:Lefschetz-torsion} to
an explicit formula for the Ray--Singer torsion and metric, twisted by
any unitary representation. This Lefschetz-type formula extends a
formula given by David Fried \cite{Fried} in the acyclic case,
i.e.~$H^*(M, \rho)= \{0\}$, via topological methods and
Reidemeister--Franz torsion. Fried interprets it as the identity
\begin{equation}
  \label{eq:4}  
  T_{RS}(M, \rho) = \bigl| \exp ( Z_F(0)) \bigr|\,,
\end{equation}
where $Z_F(0)$ stands for the analytic continuation at $s=0$ of the
dynamical function
\begin{equation*}
  Z_F(s) =  - \sum_{C} \ind(C) \tr(\rho(C)) e^{-s \ell (C)}\,.
\end{equation*}
Here the sum describes all free homotopical classes of closed orbits
of the Reeb flow $T$, $\ind (C)$ denotes its Fuller index
(Proposition~\ref{prop:ind}), $l(C)$ its length and $\rho(C)$ its
holonomy.

Our approach leads to another viewpoint on \eqref{eq:4}. Namely, we
show in Theorem~\ref{thm:torsion-zeta} and \eqref{eq:60} that our
purely spectral torsion function $\kappa(s)$ may be seen as a
dynamical zeta function, \emph{in its whole}. Indeed it holds that
\begin{equation*}
  \Gamma(s) \bigl( \kappa(s) -  \kappa(M, \rho) \bigr) = \frac{2^{1-2s}}{
    \sqrt \pi} \Gamma(\frac{1}{2} - s) Z_\rho(2s)\,,
\end{equation*}
for
$$
Z_\rho(s) = \sum_{C} \ind(C) \rtr(\rho(C)) \, \ell(C)^{s} \,,
$$
where again the sum runs over all free homotopical classes of
closed orbits of the Reeb flow, $\rtr$ is the real part of the trace,
and $\kappa(M,\rho)= 2 \dim H^0(M, \rho) - \dim H^1(M, \rho)$ is a
purely cohomological term.

This Selberg-type trace formula also has a counterpart at the level of
heat kernels. Indeed let $\tr_\kappa (e^{-t \Delta}) = 2 \tr (e^{t
  \Delta_0}) - \tr (e^{-t \Delta_1})$; then we show in
Theorem~\ref{thm:Selberg} that
\begin{equation*}
  \tr_\kappa(e^{-t \Delta}) = \dim V \frac{\sqrt \pi
    \chi(\Sigma)}{\sqrt t } + 
  \frac{1}{\sqrt{\pi t} }\sum_C 
  \ell(C) \ind(C) \rtr(\rho(C)) e^{-\ell(C)^2/ 4t}\,, 
\end{equation*}
where $\chi(\Sigma) $ is the rational Euler class of the quotient
surface orbifold $\Sigma= M / \mathbb{S}^1$.  Hence our torsion heat
trace (of fourth-order Laplacians) coincides with a dynamical theta
function. This has some surprising consequences for the small time
development of $\tr_\kappa (e^{-t \Delta})$ on CR Seifert manifolds,
but also on general $3$-dimensional contact manifolds, as given in
Corollaries~\ref{cor:full-heat} and \ref{cor:full-heat-contact}.

\smallskip

The paper is organised as follows. In \S\ref{sec:detbundle} we first
review one construction of the contact complex $(\mathcal{E}^*, d_H)$
and recall the Ray--Singer argument, from the viewpoint of the
Ray--Singer metric on the determinant of the cohomology. We then adapt
this argument to the contact complex, which leads us to define the
analytic torsion $T_C$, a torsion function $\kappa$ and a Ray--Singer
metric $\|\quad \|_C$ on $\det H (\mathcal{E}^*, d_H)$.

In \S\ref{sec:heat-kern-vari} we start the analytic study of this
torsion. After reviewing relevant properties of hypoelliptic zeta
functions and heat kernels, we establish variational formulae for
$\kappa(0)$, the torsion $T_C$ and the contact Ray--Singer metric
$\|\quad \|_C$. We then show that $\kappa(0)= 0$ in dimension $3$, and
introduce corrections of the metric $\|\quad \|_{C}$ that give CR and
contact invariants.

In \S\ref{sec:contact-ray-singer} we compare Ray--Singer analytic
torsion to ours and show that the two Ray--Singer metrics coincide on
CR Seifert manifolds.

The final \S\ref{sec:torsion-function-seifert} is devoted to the study
of the dynamical aspects of the torsion function of the contact
complex, still on CR Seifert manifolds.

\smallskip

\textbf{Acknowledgements.}  The second author thanks his PhD
supervisor Prof.~Kengo Hirachi for his expert guidance, and
Prof.~Rapha\"el Ponge for kindly explaining some of his results and
for helpful comments on an earlier draft of parts of this paper. He is
grateful to Prof.~Robin Graham for his hospitality during the second
author's visit to the University of Washington in the summer of 2006.
He thanks the organisers of the 2006 IMA Summer Program `Symmetries
and Overdetermined Systems of Partial Differential Equations' and 2006
`Seoul--Tokyo Conference in Mathematics', where parts of this project
were worked on and presented.

We are also grateful to Mike Eastwood, Robin Graham and Patrick
G{\'e}rard for useful discussions. Finally, we benefited from
stimulating and enlightening conversations with Jean-Michel Bismut,
which led us to feel some coherence or correspondence between our
results and his recent works \cite{Bismut05,BL,Bismut07} on Fried's
conjecture and the hypoelliptic Laplacian.

\section{Contact analytic torsion via a determinant bundle}
\label{sec:detbundle}

Ray and Singer \cite{RS} defined analytic torsion of the de Rham
complex as an infinite-dimensional analogue of the Reidemeister--Franz
torsion of finite simplicial complexes.  Our purpose in this section
is to describe their argument and adapt it to a similar complex
defined on contact manifolds, the construction of which we now review.

\subsection{Contact complex}
\label{subsec:contact}

Let $(M, H)$ be a smooth orientable contact manifold of dimension
$2n+1$. This means that the smooth \emph{contact distribution}
$H\subset TM$ is given as the null space of a globally defined
$1$-form, called a \emph{contact form}, satisfying the condition of
maximal non-integrability $\theta\wedge (d\theta)^n\neq 0$. The
contact forms comprise an equivalence class under multiplication by
smooth non-vanishing functions.

The \emph{contact complex} (\cite{Rumin90,Rumin94}) is a refinement of
the de Rham complex on contact manifolds defined as follows. Let
$\Omega^*M$ denote sections of the graded bundle of smooth
differential forms on $M$, $\mathcal{I}$ the ideal in $\Omega^*M$
generated by $\theta$ and $d\theta$, and $\mathcal{J}$ the ideal in
$\Omega^*M$ consisting of elements annihilated by $\theta$ and
$d\theta$. One verifies that $\mathcal{I}^k = \Omega^kM$ for $k\geq
n+1$, $\mathcal{J}^k = 0$ for $k\leq n$, and that the de Rham exterior
derivative $d$ naturally induces operators $d_H$ to form two complexes
$$
\Omega^0M\stackrel{d_H}{\longrightarrow}\Omega^1M/
\mathcal{I}^1\stackrel{d_H}{\longrightarrow}
\cdots\stackrel{d_H}{\longrightarrow}\Omega^nM/\mathcal{I}^n
$$
and
$$
\mathcal{J}^{n+1}\stackrel{d_H}{\longrightarrow}
\mathcal{J}^{n+2}\stackrel{d_H}{\longrightarrow}
\cdots\stackrel{d_H}{\longrightarrow}\mathcal{J}^{2n+1}.
$$
It is clear that these two complexes are defined independently of
the choice of $\theta$. These two complexes are joined by a
second-order differential operator $D: \Omega^nM/\mathcal{I}^n \to
\mathcal{J}^{n+1}$ defined by setting $D[\alpha] = d\beta$, where
$\beta\in\Omega^nM$ is defined by the following:

\begin{lemma}[{\cite{Rumin90,Rumin94}}]
  Let $\alpha\in\Omega^nM$. Then there exists a unique
  $\beta\in\Omega^nM$ such that $\beta\equiv\alpha$ (mod $\theta$) and
  $\theta\wedge d\beta = 0$. Moreover $d\beta \in \mathcal{J}^{n+1}$,
  and if $\alpha\in\mathcal{I}^n$ then $d\beta = 0$.
\end{lemma}

One can show that $D$ may in fact be defined independently of the
choice of $\theta$. The contact complex is
$$
\Omega^0M\stackrel{d_H}{\longrightarrow}\Omega^1M/
\mathcal{I}^1\stackrel{d_H}{\longrightarrow}
\cdots\stackrel{d_H}{\longrightarrow}\Omega^nM/
\mathcal{I}^n\stackrel{D}{\longrightarrow}
\mathcal{J}^{n+1}\stackrel{d_H}{\longrightarrow}
\mathcal{J}^{n+2}\stackrel{d_H}{\longrightarrow}
\cdots\stackrel{d_H}{\longrightarrow}\mathcal{J}^{2n+1}.
$$
We will write $\EE^k$ to mean either $\Omega^kM/\mathcal{I}^k$ or
$\mathcal{J}^k$. We also have:

\begin{prop}[{\cite[pg. 286]{Rumin94}}]
  \label{prop:contact-deRham}
  The contact complex forms a resolution of the constant sheaf $\R$
  and hence its cohomology coincides with the de Rham cohomology of
  $M$. Moreover the canonical projection $\pi: \Omega^k M \rightarrow
  \Omega^kM / \mathcal{I}^k$ for $k\leq n$ and injection $i :
  \mathcal{J}^k \rightarrow \Omega^k M$ for $k\geq n+1$ induce an
  isomorphism between the two cohomologies.
\end{prop}
The arguments being purely local, these results also apply to twisted
versions of the complex with a flat bundle, as coming from a
representation $\rho : \pi_1(M) \rightarrow U(N)$.

\smallskip

It is a basic fact that the symplectic bundle $(H, d\theta)$ admits a
contractible homotopy class of \emph{calibrated} almost complex
structures, i.e.~$J\in\End (H)$ is in this class if and only if $J^2 =
-1$ and the \emph{Levi metric} $d\theta(\cdot, J\cdot)$ is positive
definite and Hermitian.

The \emph{Reeb field} of $\theta$ is the unique vector field $T$
satisfying $\theta(T) = 1$ and $T \hook d\theta = 0$. Fixing a
$\theta$ and a $J$, we may define a Riemannian metric $g$ on $M$ by
using the Levi metric on vectors in $H$ and declaring that the Reeb
field $T$ is of unit-length and orthogonal to $H$, i.e.
$$
g= d \theta (\cdot, J \cdot) + \theta^2\,.
$$
With these choices, it is straightforward to identify the quotients
of forms appearing in the lower-half of the contact complex with
actual forms (\cite[pg. 288]{Rumin94}):
$$
\Omega^kN/\mathcal{I}^k \cong \{\alpha\in\Omega^kN : T\hook\alpha =
0 = \Lambda\alpha\},
$$
where $\Lambda$ is the adjoint of the operator $L:\Omega^kN \to
\Omega^{k+2}N$, $L\alpha = d\theta\wedge\alpha$.

We henceforth assume that $M$ is compact. With the identification
above we now have an $L^2$ inner product defined on the contact
complex. Let $\delta_H$, $D^*$ denote the formal adjoint operators. It
is straightforward to verify that
\begin{equation}
  \label{eq:5}
  \delta_H{}_{\mid \EE^k} = (-1)^k\ast d_H\ast,\quad D^\ast 
  = (-1)^{n+1}\ast D\ast,
\end{equation}
where $\ast:\EE^k\stackrel{\cong}{\rightarrow}\EE^{2n+1-k}$ is the
usual Hodge $\ast$ operator.

As a last comment here, we mention there exist other approaches to
this elementary construction of the contact complex. One possibility
is via spectral sequence considerations, using a canonical filtration
by Heisenberg weight of forms $\Omega^*M$; see \cite[\S3]{Rumin00} and
\cite{BHR} where this approach is used in the study of the
sub-Riemannian (diabatic) limit of the Hodge--de Rham spectrum.
Another interesting viewpoint is to consider the contact complex as a
curved version of a Bernstein--Gelfand--Gelfand complex in parabolic
geometry; see \cite[\S8.1]{BE} for such a presentation on the
$3$-dimensional Heisenberg group.

\subsection{Determinant bundles, metrics and Reidemeister--Franz
  torsion}

We follow the presentation of Bismut and Zhang \cite{BZ} to define the
Reidemeister--Franz torsion of a finite-dimensional complex.

Let $E$ be a finite-dimensional real vector space, and define the line
$$
\det E = \wedge^{\mathrm{max}} E\,.
$$
A useful convention here is to set $\det\{0\} = \R$ (compatible
with $\det(E \oplus F) = \det E \otimes \det F$). If $\lambda$ is a
line, let $\lambda^{-1} = \lambda^*$ be its dual line.  Then $\lambda
\otimes \lambda^{-1} = \mathrm{End}(\lambda) = \R\, \mathrm{Id}$ is
canonically isomorphic to $\R$.

One extends these notions to a finite-dimensional complex. Let
\begin{displaymath}
  (E,d)\  :\  0 \lra E_0 \stackrel{d}{\lra} E_1 \stackrel{d}{\lra} \cdots 
  \stackrel{d}{\lra} E_n \lra 0 
\end{displaymath}
be such a complex and $H^*(E,d)$ its cohomology. Define
$$
\det E = \bigotimes_{k=0}^n (\det E_k)^{(-1)^k}
$$
and
$$
\det (H^*(E,d)) = \bigotimes_{k=0}^n (\det H^k(E,d))^{(-1)^k}\,.
$$

\begin{prop}[Knudsen--Mumford~\cite{KM}]\label{prop:Knudsen-Mumford}
  The lines $\det E$ and $\det(H^*(E,d))$ are canonically isomorphic.
\end{prop}

\begin{proof} 
  We include the proof as the explicit form of the isomorphism will be
  useful below. We follow~\cite{BGSI}. Suppose first that
  $H^*(E,d)=\{0\}$ so that $\det H^*(E,d)= \R$. Then we need to find a
  canonical section of $\det E$. Let $N_k = \dim E_k$. Pick a
  non-vanishing element
  \begin{equation}
    \label{eq:6}
    s_0 = e_1 \wedge e_2 \wedge \cdots \wedge
    e_{N_0} \in \wedge^{N_0}E_0 = \det E_0\,;
  \end{equation}
  then
  \begin{equation}
    \label{eq:7}
    d s_0= de_1
    \wedge de_2 \wedge \cdots \wedge de_{N_0} \in \wedge^{N_0}E_1
  \end{equation}
  is non-vanishing since $d:E_0 \rightarrow E_1$ is injective.
  
  Next pick $s_1 \in \wedge^{N_1 - N_0} E_1$ such that $ds_0 \wedge
  s_1 $ generates $\det E_1$, and so on, taking $s_k \in \wedge^{N_i-
    \cdots +(-1)^kN_0} E_i$ such that $d s_{k-1} \wedge s_k$ generates
  $\det E_k$. Consider now
  \begin{equation}
    \label{eq:8}
    S(E,d) = s_0 \otimes (ds_0 \wedge s_1)^{-1} \otimes (ds_1 \wedge
    s_2) \otimes \cdots \otimes (ds_{n-1})^{(-1)^n} \in \det E.
  \end{equation}
  It is clear that the class $S(E,d)$ is non-zero and does not depend
  on the choices of $s_k$ for $k=0,\dots,n$, completing the proof of
  the proposition in the acyclic case.
  
  For the general case, observe that the determinants of the short
  exact sequences
  \begin{displaymath}
    \begin{matrix}
      0 & \rightarrow & dE_k &\rightarrow & \ker d{}_{\mid E_{k+1}} &
      \rightarrow & H^{k+1}(E,d)
      & \rightarrow & 0\\
      0 &\rightarrow & \ker d{}_{\mid E_{k+1}} & \rightarrow & E_{k+1}
      & \rightarrow & dE_{k+1} & \rightarrow & 0
    \end{matrix}
  \end{displaymath}
  each have a canonical element, as was just shown above. So we have
  canonical isomorphisms
  \begin{align*}
    \det (\ker d{}_{\mid E_{k+1}}) &\cong \det(dE_k) \otimes
    \det(H^{k+1}(E,d))\\
    \det (E_{k+1}) &\cong \det(\ker d{}_{\mid E_{k+1}}) \otimes
    \det(dE_{k+1}),
  \end{align*}
  and then
  $$
  \det (E_{k+1}) \cong \det(dE_k) \otimes \det(H^{k+1}(E,d))
  \otimes \det(dE_{k+1}).
  $$
  Finally taking tensor products over $k$ gives
  $$
  \det E_0 \otimes (\det E_1)^{-1} \otimes \det E_2 \otimes \cdots
  \cong \det H^0 \otimes (\det H^1)^{-1} \otimes \det H^2 \otimes
  \cdots
  $$
  canonically.
\end{proof}

Suppose now $E$ is given a metric $g$. Hence $\det E$ has an induced
metric. One can then define a metric on $\det H^*(E,d)$ by
$$
\|\quad \|_{\det H^*(E,d)} = \|\quad \|_{\det E},
$$
using the canonical isomorphism given by
Proposition~\ref{prop:Knudsen-Mumford}.

Let $d^* = \delta$ be the adjoint of $d$. By finite-dimensional Hodge
theory, $H^*(E,d)$ identifies with the harmonic forms
$$
\mathcal{H}^*(E,d) = \{ s \in E \mid ds = d^* s = 0\}\,.
$$
By their inclusion in $E$, the harmonic forms inherit a metric. We
then have a a second metric $| \quad|_{\det H^*(E,d)}$ on $\det
H^*(E,d)$ via the above identification.

\begin{defn}\label{def:metr-franz-reid}
  The \emph{torsion} of the complex $(E,d)$ with metric $g$ is the
  ratio
  $$
  \tau(E,d,g) = \frac{\|\quad \|_{\det H^*(E,d)}}{ |\quad |_{\det
      H^*(E,d)}}\,.
  $$
\end{defn}
\begin{rem}
  \label{rem:R-torsion}
  Note that this definition of torsion, given in \cite[\S2]{BZ} for
  instance, is quite natural at the metric level, but actually leads
  to the \emph{inverse} of the original Reidemeister--Franz torsion
  (or $R$-torsion), i.e.
  \begin{displaymath}
    \tau_R (E,d,g) = 1/\tau(E,d,g)\,;
  \end{displaymath}
  see \cite[\S1]{RS} and \cite[\S2]{Fried}.
\end{rem}
One can be more explicit using the proof of
Proposition~\ref{prop:Knudsen-Mumford}. Consider
$$
F= \mathcal{H}^*(E,d)^\bot\,.
$$
The complex $(F,d)$ is acyclic so we can construct the canonical
class $S(F,d)$ as in \eqref{eq:8}.
\begin{prop}\label{prop:formula_tau}
  For any generators $s_i\in\det (\ker d)^\bot$ we have
  \begin{equation}
    \label{eq:9}
    \tau(E,d,g) = \|S(F,d)\|_{\det F} = \Bigl(\frac{\|s_0\|}{
      \|d s_0\|}\Bigr) \times 
    \Bigl(\frac{\|s_1\|}{ \|d s_1\|}\Bigr)^{-1} \times \cdots \times
    \Bigl(\frac{\|s_{n-1}\|}{ \|d s_{n-1}\|}\Bigr)^{(-1)^{n-1}} \,.
  \end{equation}
  
\end{prop}

\begin{proof}
  The splitting $E = \mathcal{H}^*(E,d) \oplus F$ induces the
  canonical isomorphism
  \begin{align*}
    \det \mathcal{H}^*(E,d) & \stackrel{\cong}{\lra} \det E = \det
    \mathcal{H}^*(E,d)
    \otimes \det F \\
    s & \longmapsto s \otimes S(F,d)\,.
  \end{align*}
  Then
  $$
  \|s \otimes S(F,d)\|_{\det E} = |s|_{\det H^*(E,d)}
  \|S(F,d)\|_{\det F}\,
  $$
  and by Definition~\ref{def:metr-franz-reid}
  \begin{align*}
    \tau(E,d,g) & = \|S(F,d)\|_{\det F} \\
    & = \|s_0\|\times \|ds_0 \wedge s_1\|^{-1}\times \|ds_1 \wedge
    s_2\| \times \cdots \times  \|ds_{n-1}\|^{(-1)^n}\\
    & = \Bigl(\frac{\|s_0\|}{ \|ds_0\|}\Bigr) \times
    \Bigl(\frac{\|s_1\|}{ \|ds_1\|}\Bigr)^{-1} \times \cdots \times
    \Bigl(\frac{\|s_{n-1}\|}{ \|ds_{n-1}\|}\Bigr)^{(-1)^{n-1}} \,,
  \end{align*}
  since if $s_i\in \det (\ker d)^\bot$, $\|d s_i \wedge s_{i+1}\| =
  \|d s_i\| \|s_{i+1}\|$.
\end{proof}

At this point in the Riemannian case (\cite{RS, BZ}) one can guess the
correct formula for analytic torsion by considering the
Reidemeister--Franz torsion of finite-dimensional subcomplexes that
approximate the infinite-dimensional de Rham complex $(\Omega^*, d)$.
A natural choice of subcomplexes is obtained here by taking cut-off de
Rham complexes using the spectrum of the Hodge--de Rham Laplacian
$\Delta = d\delta + \delta d$; that is, one considers the energy
levels $\Omega^*_{[0,\lambda]} = \{\Delta \leq \lambda\}$. One then
expresses \eqref{eq:9} using the determinants of $\Delta$ on
$(\Omega^*_{[0,\lambda]},d)$ and finally as combinations of
differentiated zeta functions $\zeta'(\Delta)(0)$ in the limiting
infinite-dimensional case. We carry out this procedure for the contact
complex next.

\subsection{Defining a contact analytic torsion}

Consider now the contact complex
\begin{displaymath}
  (\EE,d_H)\  :\  \EE^0 \stackrel{d_H}{\lra} \EE^1
  \stackrel{d_H}{\lra} \cdots  
  \stackrel{d_H}{\lra} \EE^n \stackrel{D}{\lra}
  \EE^{n+1}\stackrel{d_H}{\lra} \cdots \stackrel{d_H}{\lra} \EE^{2n+1}. 
\end{displaymath}
We want to define finite-dimensional subcomplexes of the contact
complex via finite energy cut-offs for a certain Laplacian $\Delta$.

We shall use the following uniformly fourth-order Laplacian:
\begin{equation}
  \label{eq:10}
  \Delta := 
  \begin{cases}
    (d_H\delta_H + \delta_Hd_H)^2 & \textnormal{ on $\EE^k$ for }
    k\neq n,n+1\\
    (d_H\delta_H)^2 + D^\ast D & \textnormal{ on }\EE^n\\
    DD^\ast + (\delta_Hd_H)^2 & \textnormal{ on }\EE^{n+1}.
  \end{cases}
\end{equation}

We denote by $\Delta_k$ the restriction of $\Delta$ to $\EE^k$. The
rationale behind our choice of $\Delta$ is as follows. In middle
degrees, because $D$ is second-order, one needs to square the terms
involving $d_H$ so that $\Delta$ has certain good analytical
properties. In particular, $\Delta$ is maximally hypoelliptic and
invertible in the Heisenberg symbolic calculus, while the standard
combinations $d_H\delta_H + D^* D$ and $DD^* + d_H \delta_H$ are not;
see \S\ref{subsec:heisenberg} below. If we then consider the spectral
spaces $E^n_{[0, \lambda]} = \{ \Delta_n \leq\lambda\}$ as successive
finite-dimensional approximations of $\EE^n$, in order to include
$E^n_{[0,\lambda]}$ in a finite-dimensional subcomplex of $(\EE^*,d)$
we need the Laplacians outside middle degree to be fourth-order also.

\begin{rem}
  Note that $\Delta$ is different to the Laplacian $\Delta_Q$ defined
  in~\cite{Rumin94}. The latter was defined with nice algebraic
  properties, namely commutativity with $J$ when $\LL_T J = 0$.  On
  the other hand, observe that $\Delta$ commutes with $d_H, D$ and
  their adjoints, which, as we shall see in
  \S\ref{sec:vari-behav-analyt}, is essential for analytic torsion of
  the contact complex having the correct variational behaviour.
  Moreover note that $\Delta$ is the Laplacian appearing in the
  sub-Riemannian limit (\cite{Rumin00}).
\end{rem}

Next let us set
$$
E^*_{[0,\lambda]} = \bigoplus_{k= 0}^{2n+1}\{\Delta_k \leq \lambda
\} \,.
$$
These are finite-dimensional subcomplexes of the contact complex,
each of which splits into a direct sum of eigenspaces $E^*_\mu$ of
$\Delta$ for $0\leq \mu \leq \lambda$.

\begin{prop}
  \label{prop:Reidemeister-contact}
  The torsion of $(E^*_{[0,\lambda]}, d_H)$ is
  \begin{equation}
    \label{eq:11}
    \tau (E^*_{[0,\lambda]}, d_H) = \prod_{k=0}^{2n+1} \det
    (\Delta_k\mid E^*_{]0, 
      \lambda]})^{(-1)^k w(k)/4}\,,
  \end{equation}
  where
  \begin{equation}
    \label{eq:12}
    w(k) =
    \begin{cases}
      k & \mathrm{if}\quad k \leq n\\
      k+1 & \mathrm{if} \quad k > n\,.
    \end{cases}
  \end{equation}
\end{prop}

\begin{rem}
  This $w(k)$ is the natural weight of $\EE^k$ in the contact complex;
  see \S\ref{subsec:heisenberg} and \cite[\S3]{Rumin00}.
\end{rem}

\begin{proof}
  We use Proposition \ref{prop:formula_tau}.  By the orthogonal
  splitting $\dsp E^*_{[0,\lambda]} = \oplus_{\mu \leq \lambda}
  E_\mu^*$ and \eqref{eq:9} we have
  \begin{equation}
    \label{eq:13}
    \tau (E^*_{[0,\lambda]}, d_H) = \prod_{0< \mu \leq \lambda}
    \|S(E_{\mu}, d_H)\| .
  \end{equation}
  Given $\mu \in ]0, \lambda]$ we pick an orthogonal basis $(v_1,
  \cdots, v_{N_k})$ of each eigenspace
  $$
  E^k_\mu \cap (\ker d)^\bot = F_\mu^k
  $$
  and choose the elements $ s_i = v_1 \wedge \cdots \wedge v_{N_k}
  \in \det F_\mu^k$ as in \eqref{eq:6}. Here $N_k = \dim F_\mu^k$. Now
  by \eqref{eq:7} and orthogonality
  \begin{align*}
    \|ds_i\| & = \|dv_1 \wedge dv_2 \wedge \cdots dv_{N_k}\|\\
    & =
    \|dv_1\|\times \|dv_2\| \times \cdots \times \|dv_{N_k} \|\\
    & =
    \begin{cases}
      \mu^{N_k/4} \|v_1\| \times \cdots \times\|v_{N_k}\| =
      \mu^{N_k/4}\| s_i\|&
      \textnormal{ for }k \not=n  \\
      \mu^{N_k/2} \|v_1\| \times \cdots \times\|v_{N_k}\| =
      \mu^{N_k/2} \|s_i\| &\textnormal{ for } k=n\,.
    \end{cases}
  \end{align*}
  
  Therefore from \eqref{eq:9}
  \begin{equation}
    \label{eq:14}
    \|S(E_{\mu}, d_H)\| = \mu^{M_\mu}
  \end{equation}
  with
  \begin{equation*}
    4M_\mu  = - N_0 + N_1- N_2 \cdots + (-1)^{n+1} 2 N_n + (-1)^n
    N_{n+1} \cdots - N_{2n} \,.
  \end{equation*}
  Let $e_k = \dim E^k_\mu$. From $E^k_\mu = F^k_\mu \oplus d
  F^{k-1}_\mu$, one has $e_k = N_k + N_{k-1}$ and therefore
  $$
  N_k = e_k - e_{k-1} + \cdots+ (-1)^k e_0.
  $$
  This leads to
  \begin{align*}
    4 M_\mu & = -(2n+2) e_0 + (2n+1) e_1 + \cdots + (-1)^n (n+3)
    e_{n-1}
    \\
    & \quad +(-1)^{n+1} (n+2) e_n + (-1)^{n+2} n e_{n+1} \\
    & \quad + (-1)^{n+3} (n-1) e_{n+2} + \cdots + e_{2n+1}\\
    & = \sum_{k=0}^{2n+1} (-1)^k w(k) e_k \,,
  \end{align*}
  using that $\dsp \sum_{k=0}^{2n+1} (-1)^k e_k =0$ for the acyclic
  complex $(E_\mu^*,d_H)$.
  
  As $\Delta = \mu$ on $E^*_\mu$, equation \eqref{eq:14} reads
  $$
  \|S(E_\mu, d_H)\|^4 = \prod_{k=0}^{2n+1} \det (\Delta_{k} \mid
  E_\mu)^{(-1)^k w(k)},
  $$
  leading by \eqref{eq:13} to the assertion of the proposition.
\end{proof}

We finally introduce zeta functions of the contact Laplacian. If
$\spec^*(\Delta_k)$ denotes the non-zero spectrum of $\Delta_k$ on
$\EE_k$, then we take
$$
\zeta (\Delta_k)(s) = \dim H^k(\EE,d_H) + \sum_{\lambda \in
  \spec^*(\Delta_k) } \lambda^{-s}.
$$
Note that by hypoellipticity (or
Proposition~\ref{prop:contact-deRham}) $\dim H^k(\EE,d_H)$ is finite.
By the results in \S\ref{subsec:heisenberg} below, $\zeta
(\Delta_k)(s)$ admits a meromorphic extension to $\C$ that is regular
at $s=0$. On each subcomplex $(E^*_{[0,\lambda]},d_H)$ we then have
\begin{align*}
  \zeta'(\Delta_k \mid E_{]0, \lambda]})(0) & = - \sum_{\mu \in
    \spec^*(\Delta_k)\cap ]0, \lambda]} \ln \mu \\
  & = - \ln\det (\Delta_k \mid E_{]0, \lambda]}) \,.
\end{align*}
Thus formula \eqref{eq:11} for the torsion of $(E^*_{[0,
  \lambda]},d_H)$ can be written for $\lambda >0$ as
\begin{equation}
  \label{eq:15}
  \ln \tau (E^*_{[0, \lambda]},d_H) = 
  \frac{1}{4}\sum_{k=0}^{2n+1} (-1)^{k+1}
  w(k) \zeta'(\Delta_k \mid E^*_{]0, \lambda]})(0)\,.
\end{equation}
We thus speculate in extending this formula to the whole contact
complex by defining the \emph{analytic torsion of the contact complex}
as
\begin{equation}
  \label{eq:16}
  \ln T_{C} = \frac{1}{4} \sum_{k=0}^{2n+1} (-1)^{k+1}
  w(k) \zeta'(\Delta_k)(0)\,.
\end{equation}
This formula is very similar to that of Ray--Singer analytic torsion
$T_{RS}$ in the Riemannian setting. Namely, from \cite[Definition
1.6]{RS}, in dimension $N$
\begin{equation}
  \label{eq:17}
  \ln T_{RS} =  \frac{1}{2}\sum_{k=0}^N (-1)^k
  k \zeta'(\Delta_k)(0)\,,
\end{equation}
for Hodge--de Rham Laplacians $\Delta_k$. Note however the sign
convention: $T_{RS}$ coincides with Reidemeister--Franz torsion
$\tau_R$ on finite-dimensional cut-off de Rham complexes, while our
$T_C$ leads to the inverse; see Remark~\ref{rem:R-torsion}.

\smallskip

By analogy with Definition~\ref{def:metr-franz-reid}
and~\cite{Quillen,BZ}, we also define a \emph{contact complex
  Ray--Singer metric} on $\det H^*(\EE,d_H)$ by setting
\begin{equation}
  \label{eq:18}
  \|\quad \|_C = T_C \mid \quad |_{L^2(\mathcal{E})}\,.
\end{equation}
Here $ |\quad |_{L^2(\mathcal{E})}$ is the $L^2$ metric induced on
$\det H^*(\mathcal{E},d_H)$ by identification of $H^*(\EE,d_H)$ with
harmonic forms $\mathcal{H}^*(\EE,d_H) \subset \EE^*$.

Again, note that the Ray--Singer metric on the de Rham determinant
$\det H^*(\Omega^*M, d)$ reads instead
\begin{equation}
  \label{eq:19}
  \|\quad \|_{RS} = (T_{RS})^{-1} |\quad |_{L^2(\Omega^*M)}\, .
\end{equation}

\smallskip

More generally, we can twist the contact complex with a flat bundle
and then define the analytic torsion of this twisted contact complex
$(\mathcal{E}^*_\rho, d_H)$.  Indeed let $\rho : \pi_1(M)\to U(N)$ be
a unitary representation on $\C^N$. Associated to $\rho$ is an
Hermitian complex rank $N$ vector bundle $V_\rho$ equipped with a
canonical metric-preserving flat connection $\nabla_\rho$. One sets
$\mathcal{E}_\rho = \mathcal{E} \otimes V_\rho$ and $d_H(\alpha
\otimes s) = d_H \alpha \otimes s$ for parallel $s$. From this we may
define the contact analytic torsion $T_C(\rho)$ with associated
contact complex Ray--Singer metric on $\det H^*(\EE_{\rho},d_H)$.

The conciseness of notations $T_C$ and $T_C(\rho)$ should not be
misleading. The (twisted) contact complex only depends on the contact
structure $H$ on $M$ (and $\rho$), but the spectral invariants $T_C$
and $T_C(\rho)$ also depend on the choices of a contact form $\theta$
and complex structure $J$, both being used in the metric $g$.

\smallskip

Since we have defined this analytic torsion through algebraic and
formal considerations around Reidemeister--Franz torsion, we now need
to study its analytical properties. That is the purpose of the next
section.

\section{Heat kernels and variational behaviour of the torsion}
\label{sec:heat-kern-vari}

We first gather some properties of zeta functions and the heat
development of hypoelliptic operators such as the Laplacian of the
contact complex.

\subsection{Heat kernels and zeta functions for hypoelliptic
  operators}
\label{subsec:heisenberg}

The Laplacian $\Delta$ for the contact complex is not elliptic.
However there is a (substantially more intricate) symbolic calculus
that can be applied to it to obtain results on heat kernels
qualitatively analogous to the elliptic case. This calculus is called
the \emph{Heisenberg calculus} and was introduced by
Beals--Greiner~\cite{BG} and Taylor~\cite{Taylor}. A short account of
its properties may be found in \cite{Getzler}, and its use for the
contact complex has been presented by Julg and Kasparov in
\cite[\S5]{JK}. This calculus has also been developed in a more
general setting that includes the contact case; see the presentation
of Ponge~\cite{PongeAMS} for details and further references to the
literature. Here we just briefly sketch the results that we shall need
in the sequel.

\begin{prop}[{Ponge~\cite[Ch.~5]{PongeAMS}}]
  \label{prop:hypoheat}
  Let $\VV$ be a vector bundle over a compact contact manifold $(M,H)$
  of dimension $2n+1$. Let $P:C^\infty(M,\VV)\to C^\infty(M,\VV)$ be a
  differential operator of even Heisenberg order $v$ that is
  selfadjoint and bounded from below. If $P$ satisfies the Rockland
  condition at every point then the principal symbol of $P +
  \partial_t$ is an invertible Volterra--Heisenberg symbol and as
  $t\searrow 0$ the heat kernel $k_t(x,x)$ of $P$ on the diagonal has
  the following asymptotics in $C^\infty(M,(\End
  \VV)\otimes|\Lambda|(M))$:
  \begin{equation*}
    k_t(x,x) \sim \sum_{j=0}^\infty t^{\frac{2(j-n-1)}{v}}a_j(P)(x).
  \end{equation*}
\end{prop}

Some explanation about the proposition is in order. The
\emph{Heisenberg order} of $P$ is defined by assuming that a
derivative in the direction of the Reeb field $T$ has weight 2, while
derivatives in the direction of the contact distribution $H$ have
weight 1. The \emph{Rockland condition} is a representation-theoretic
condition defined in~\cite[Definition 3.3.8]{PongeAMS}. (The original
formulation is due to Rockland~\cite{Rockland}.)
An operator that satisfies this condition is hypoelliptic, in the
sense of \cite[Proposition 3.3.2]{PongeAMS}. Invertibility of an
operator in the \emph{Volterra--Heisenberg calculus} is explained
in~\cite[Ch.~5]{PongeAMS}.

The next result describes the properties of the zeta function in this
contact setting.  For the non-negative operators $P$ we are concerned
with, the result follows from Proposition \ref{prop:hypoheat} by a
classical argument using the Mellin transform of the heat kernel.

\begin{prop}[{Ponge~\cite[\S4]{PongeJFA}}]
\label{prop:hypozeta}
Let $P$ be as in Proposition~\ref{prop:hypoheat}. Then the zeta
function
$$
\zeta(P)(s) = \dim\ker P + \tr^* (P^{-s})\ , s\in\C,
$$
is a well-defined holomorphic function for $\Re(s)\gg 1$ and admits
a meromorphic extension to $\C$ with at worst simple poles occurring
at $s \in \mathcal{S}= \{\frac{2(n+1-j)}{v} \mid j \in \N\} \setminus
(-\N)$.  Moreover
$$
\zeta(P)(0) = \int_M \tra(a_{n+1}(P))\, \theta\wedge (d\theta)^n\,
$$
is the constant term in the development of $\tr (e^{-tP})$ as $t
\searrow 0$.
\end{prop}

Now by~\cite[p.~300]{Rumin94} and the definition (\cite[Definition
3.3.8]{PongeAMS}) of the Rockland condition, or else by Julg and
Kasparov's work \cite[\S5]{JK}, the fourth-order Laplacian $\Delta$ on
the contact complex (twisted with a flat bundle) satisfies the
Rockland condition, hence Propositions~\ref{prop:hypoheat}
and~\ref{prop:hypozeta} apply to it.

\begin{rem}
  Our convention is to include the eigenvalue $0$ in the definition of
  the zeta function of a non-negative hypoelliptic operator $P$. This
  means $ \zeta(P)(0)$ consists of (the integral of) purely local
  terms for $P$.
\end{rem}

\subsection{Variational behaviour of the analytic torsion}
\label{sec:vari-behav-analyt}

We consider the variation of contact analytic torsion and Ray--Singer
metric in the direction of an arbitrary line of pairs
$(\theta_\varepsilon, J_\varepsilon)$ of contact form and calibrated
almost complex structure for $(M,H)$.  We shall see that the variation
of the Ray--Singer metric is given entirely by local terms.  Indeed
this may be viewed as a necessary and sufficient condition for
correctly defining an analytic torsion; see the approach of
Branson~\cite{Branson}.

First we define from \eqref{eq:16} the \emph{contact torsion function}
by
\begin{equation}
  \label{eq:20}
  \kappa(s) =
  \frac{1}{2}\sum_{k=0}^{2n+1}(-1)^{k+1}w(k)\zeta(\Delta_k)(s),
\end{equation}
with $w(k)$ as in \eqref{eq:12}. Then the analytic torsion of the
contact complex reads
\begin{equation*}
  T_C = \exp\Bigl( \frac{1}{2}\kappa'(0)\Bigr).
\end{equation*}

For simplicity, in this section we suppress the $C$ from the notation,
as well as the representation $\rho$, although all results stand for
the twisted torsions and metrics as well.

\begin{thm}
  \label{thm:contact-tor-var}
  Let a $\bullet$ superscript denote first variation
  $(d/d\varepsilon)|_{\varepsilon=0}$.
  \begin{enumerate}
  \item One has $\kappa(0)^\bullet = 0$, so that $\kappa(0)$ is a
    contact invariant.
  \item The variation of the analytic torsion $T$ is given by
    \begin{equation*}
      (\ln T)^\bullet = \sum_{k=0}^n (-1)^k\left(\int_M\tra(\alpha \,
        a_{t^0;\,k})\,\dvol - \tr(\alpha\PP_k) \right), 
    \end{equation*}
    where $\alpha = \ast^{-1}\ast^{\bullet}$, $a_{t^0;\,k}$ is the
    $t^0$ coefficient in the diagonal small-time asymptotic expansion
    of the heat kernel of $\Delta_k$, $\dvol$ is the volume form
    $\theta\wedge (d\theta)^n$, and $\PP_k$ is orthogonal projection
    onto the null-space of $\Delta_k$.
  \end{enumerate}
\end{thm}
\begin{proof}
  By Hodge $\ast$ duality (for convenience we suppress the
  $\varepsilon$ dependence)
  \begin{equation*}
    \kappa(s) = \sum_{k=0}^nc_k\zeta(\Delta_k)(s)
  \end{equation*}
  with
  $$
  c_k = (-1)^k(n+1-k).
  $$
  By a Mellin transform set
  $$
  f(s) = \Gamma(s)\kappa(s) = \sum_{k=0}^{n} c_k\int_0^{+\infty}
  t^{s-1}\tr(e^{-t\Delta_k} - \PP_k)dt + \Gamma(s)\sum_{k=0}^{n}
  c_k\dim H^k(\EE,d_H),
  $$
  where $\PP_k$ denotes orthogonal projection onto the null-space
  of $\Delta_k$.
  \begin{lemma}
    \label{lemma:duhamel}
    It holds that
    $$
    (\tr(e^{-t\Delta}))^\bullet = -t\tr(\Delta^\bullet
    e^{-t\Delta}).
    $$
  \end{lemma}
  \begin{proof}
    Duhamel's formula (see e.g.~\cite[Proposition 3.15]{rosenberg})
    implies that
    \begin{equation*}
      \tr (e^{-t(\Delta + \varepsilon\Delta^\bullet + O(\varepsilon^2))}) -
      \tr (e^{-t\Delta}) = -\int_0^t \tr(e^{-(t-s)(\Delta +
        \varepsilon\Delta^\bullet +
        O(\varepsilon^2))}\varepsilon(\Delta^\bullet +
      O(\varepsilon))e^{-s\Delta}) \,ds. 
    \end{equation*}
    Dividing both sides by $\varepsilon$ then letting $\varepsilon\to
    0$ gives
    $$
    (\tr(e^{-t\Delta}))^\bullet = -\int_0^t
    \tr(e^{-(t-s)\Delta}\Delta^\bullet e^{-s\Delta}) \ ds.
    $$
    Recalling that $\tr(AB) = \tr(BA)$, for smoothing operators $A,
    B$, we have
    $$
    (\tr(e^{-t\Delta}))^\bullet = -\int_0^t \tr(\Delta^\bullet
    e^{-s\Delta}e^{-(t-s)\Delta}) \ ds,
    $$
    and the lemma follows.
  \end{proof}
  
  By Lemma~\ref{lemma:duhamel}
  \begin{equation}
    \label{eq:fsbullet}
    f(s)^\bullet = -\sum_{k=0}^{n} c_k\int_0^{+\infty}
    t^s\tr(\Delta_k^\bullet e^{-t\Delta_k}), 
  \end{equation}
  since by hypoellipticity $\tr(\PP_k) = \dim\ker\Delta_k = \dim
  H^k(\EE,d_H)$ is certainly independent of $\theta$ and $J$.
  
  Setting $\alpha = \ast^{-1}\ast^{\bullet}$, one computes using
  \eqref{eq:5} the variation of the Laplacian as
  $$
  \Delta^\bullet = \left\{
    \begin{array}{ll} -d_H\alpha\delta_Hd_H\delta_H +
      d_H \delta_H \alpha d_H\delta_H - d_H\delta_Hd_H\alpha\delta_H +
      d_H\delta_H d_H \delta_H\alpha\\
      \quad -\alpha \delta_H d_H \delta_H d_H + \delta_H\alpha
      d_H\delta_Hd_H
      - \delta_H d_H\alpha\delta_H d_H + \delta_H d_H\delta_H\alpha d_H\\
      \quad\quad \textnormal{ on $\EE^k$ for
      }k=0,\dots,n-1 \,,\vspace{2mm}\\
      -d_H\alpha\delta_Hd_H\delta_H + d_H\delta_H\alpha d_H\delta_H -
      d_H\delta_Hd_H\alpha\delta_H +
      d_H\delta_Hd_H\delta_H\alpha\\
      \quad -\alpha D^\ast D + D^\ast\alpha D\\
      \quad\quad \textnormal{ on }\EE^n \,.
    \end{array} 
  \right.
  $$
  A computation then shows that
  \begin{equation}
    \label{eq:comb}
    \begin{split}
      \sum_{k=0}^{n}c_k\tr(\Delta_k^\bullet e^{-t\Delta_k}) &=
      2\sum_{k=0}^{n-1} \tr\left(\left(\alpha(c_k +
          c_{k-1})(d\delta)^2 - \alpha(c_k + c_{k+1})(\delta
          d)^2\right) e^{-t\Delta_k}
      \right)\\
      &\quad + 2\tr\left(\left(\alpha(c_n + c_{n-1})(d\delta)^2 -
          \alpha c_n D^*D\right)e^{-t\Delta_n}\right).
    \end{split}
  \end{equation}
  To move the $\alpha$'s to the front we have used the following
  facts: the heat kernel is a semigroup, implying $e^{-t\Delta} =
  e^{-(t/2)\Delta}e^{-(t/2)\Delta}$; if operators $A, B$ are smoothing
  then $\tr(AB) = \tr(BA)$; and the Laplacian $\Delta$ commutes with
  $d_H, D$ and their adjoints. We have also used that
  $$
  \tr(\alpha DD^\ast e^{-t\Delta_{n+1}}) = -\tr(\alpha D^\ast D
  e^{-t\Delta_{n}}),
  $$
  as $\alpha DD^\ast e^{-t\Delta_{n+1}} = - *^{-1}(\alpha D^\ast D
  e^{-t\Delta_{n}}) *$, which follows from (\ref{eq:5}) and
  $(*^2)^\bullet = * \alpha + \alpha * =0$.
  
  Simplifying \eqref{eq:comb} yields
  \begin{equation*}
    \begin{split}
      \sum_{k=0}^{n}c_k\tr(\Delta_k^\bullet e^{-t\Delta_k})
      &= 2\sum_{k=0}^n(-1)^{k+1}\tr(\alpha\Delta_ke^{-t\Delta_k})\\
      &= -2\frac{d}{dt}\sum_{k=0}^n (-1)^{k+1}\tr (\alpha
      e^{-t\Delta_k}).
    \end{split}
  \end{equation*}
  Hence after integrating by parts in \eqref{eq:fsbullet},
  \begin{equation}
    \label{eq:fsdot}
    f(s)^\bullet = 2s\sum_{k=0}^n (-1)^k\int_0^{+\infty} t^{s-1}\tr(\alpha
    e^{-t\Delta_k} - \alpha\PP_k)dt\,. 
  \end{equation}
  We have included the $\alpha\PP_k$ term in the integrand to ensure
  that $f(s)^\bullet$ is a well-defined meromorphic function on all
  $\C$.
  
  As $1/\Gamma(s) = O(s)$, $f(s)^\bullet$ is regular at the origin.
  Now near $s=0$
  $$
  f(s) = \Gamma(s)(\kappa(0) + s\kappa^\prime(0) + O(s^2))
  $$
  so
  $$
  f(s)^\bullet = \Gamma(s)\kappa(0)^\bullet +
  s\Gamma(s)(\kappa^\prime(0))^\bullet + O(s).
  $$
  Since $f(s)^\bullet$ is regular at the origin, letting $s\to 0$
  in the above equation gives
  \begin{equation*}
    \kappa(0)^\bullet = 0 \quad \textnormal{ and }\quad
    (\kappa^\prime(0))^\bullet = 
    f(0)^\bullet. 
  \end{equation*}
  This gives assertion (1) of Theorem~\ref{thm:contact-tor-var}.
  Assertion (2) follows from Proposition~\ref{prop:hypozeta},
  after taking the inverse Mellin transform of \eqref{eq:fsdot}.
\end{proof}

\begin{rem}
  \label{rem:unicity}
  The previous proof actually shows that the torsion function
  $$
  \kappa(s)= \sum_{k=0}^n (-1)^k (n+1-k) \zeta(\Delta_k)(s)
  $$
  studied here is, up to a multiplicative factor, the \emph{unique}
  combination of such zeta functions that leads to a variational
  formula like \eqref{eq:fsdot}, i.e.~local up to cohomological terms.
\end{rem}

\medskip

The variational formula for analytic torsion we obtained is more
neatly expressed at the level of the Ray--Singer metric, since then
the global term disappears. The next result is analogous to the
variational formula for the Ray--Singer metric on Riemannian manifolds
(see~\cite[Theorem 4.14]{BZ} and~\cite[Theorem 1.18]{BGSIII}).

\begin{cor}
  \label{cor:contact-quillen-var}
  Let $\|\quad \|_{C}$ denote the contact Ray--Singer metric on $\det
  H^*(\mathcal{E}, d_H)$.
  \begin{enumerate}
  \item The following identity holds:
    \begin{equation}
      \label{eq:contact-quillen-var} 
      (\ln\|\quad \|_{C})^\bullet = \sum_{k=0}^n
      (-1)^k \int_M \tra(\alpha a_{t^0;\,k})\, \theta \wedge (d
      \theta)^n \,.  
    \end{equation}
  \item Under conformal variations of the contact form
    $(\theta_\varepsilon = e^{2\varepsilon \Upsilon }\theta,
    J_\varepsilon = J)$, for a function $\Upsilon$, we have
    \begin{equation*}
      (\ln\|\quad \|_{C})^\bullet = 2\sum_{k=0}^n
      (-1)^k(n+1-k)\int_M \Upsilon \tra(a_{t^0;\,k}) \, \theta\wedge
      (d\theta)^n \,.
    \end{equation*}
  \end{enumerate}
\end{cor}
\begin{proof}
  Recalling from \eqref{eq:18} and \S\ref{sec:detbundle} the
  definition of the Ray--Singer metric, we have that
  \begin{equation}
    \label{eq:metric-dot}
    \ln\|\quad \|^2_{C} = 2\ln T_C + \ln|\quad
    |^2_{L^2(\mathcal{E})}\,,
  \end{equation}
  where the $L^2$ metric $|\quad |_{L^2(\mathcal{E})}$ is induced on
  $\det H^*(\EE,d_H)$ from the inner product on $H^*(\EE,d_H)$ defined
  by
  $$
  \langle [u],[v]\rangle_{\theta,\,J} = \int_M \PP u \wedge\ast\PP
  v.
  $$
  But for the orthogonal projection $\PP$ onto harmonic forms
  $\mathcal{H}^*(\EE,d_H)$, one checks that $\PP^\bullet$ takes
  $\mathcal{H}^*(\EE,d_H)$ to its orthogonal complement. Thus
  \begin{equation}
    \label{eq:inner-prod-dot}
    \langle [u],[v]\rangle^\bullet = \int_M \PP u \wedge(\ast^\bullet)\PP
    v = \langle [u],\alpha[v]\rangle. 
  \end{equation}
  
  If we use Hodge $\ast$ duality in the definition of $\det
  H^*(\EE,d_H)$, then take an orthonormal basis of each
  $\mathcal{H}^k(\EE,d_H)$, $k=0,\dots,n$, and finally use
  \eqref{eq:inner-prod-dot}, it is easy to see that
  $$
  (\ln|\quad |_{C})^\bullet = \sum_{k=0}^{n}
  (-1)^k\tr(\alpha\PP_k).
  $$
  This together with \eqref{eq:metric-dot} and
  Theorem~\ref{thm:contact-tor-var} completes the proof of (1).
  
  \smallskip Assertion (2) follows immediately from (1), since for
  conformal variations it is straightforward to check that on $\EE^k$
  $$
  \alpha = \ast^{-1}\ast^\bullet = 2(n+1-k)\Upsilon\id \,.
  $$
\end{proof}

Note that setting $\Upsilon \equiv 1$ in
Corollary~\ref{cor:contact-quillen-var} (2), i.e.~performing a
constant rescaling $\theta \mapsto e^{2\varepsilon} \theta$, yields
\begin{equation}
  \label{eq:21}
  \bigl( \|\quad \|_C \bigr)^\bullet = 2 \kappa(0)\,. 
\end{equation}

In particular, if the contact invariant $\kappa(0)\neq 0$, then we
could not hope for any invariance of Ray--Singer metric. Note that by
definition
\begin{align*}
  \kappa(0 ) & = \sum_{k=0}^n (-1)^k (n+1-k) \zeta(\Delta_k)(0) \\
  & = \sum_{k=0}^n (-1)^k (n+1-k) \int_M \tra(a_{t^0;\,k})\, \theta
  \wedge (d \theta)^n\,,
\end{align*}
by Proposition~\ref{prop:hypozeta}, where again $a_{t^0;\,k}$ is the
constant $t^0$ coefficient in the diagonal small-time asymptotic
expansion of the heat kernel of $\Delta_k$.  Therefore $\kappa(0)$ is
an integral over $M$ of local curvature data, namely
\begin{equation}
  \label{eq:22}
  \kappa(0) = \int_M P_n(R, A, T, \nabla) \dvol \,,
\end{equation}
for some universal polynomial in Tanaka--Tanno--Webster (\cite{Tanaka,
  Tanno, Webster}) curvature $R$, torsion $A$, Tanno's tensor $T=
\nabla J$, and their covariant derivatives.

We show in Corollary~\ref{cor:3var} below that, in dimension $3$
($n=1$), $\kappa(0)$ vanishes identically. Whether contact invariants
such as $\kappa(0)$ vanish in all dimensions is an open problem. For
further discussion in the contact case see~\cite[\S7]{Seshadri-ACH},
and \cite[Remark 9.3]{BHR} for a similar problem arising for the eta
function of the contact complex.

\subsection{CR/contact invariants in dimension $3$} 
\label{sec:cont-case-dimens}

In dimension $3$, besides the vanishing of $\kappa(0)$ we mentioned,
we can also obtain more explicit variational formulae, and get
CR/contact-invariant corrections to the contact Ray--Singer metric.
\begin{cor}
\label{cor:3var}
On $3$-dimensional contact manifolds:
\begin{enumerate}
\item It holds that $\kappa(0) = 0$, and thus $\|\quad \|_C$ is
  independent of a constant rescaling $\theta\mapsto K \theta$.
\item There exist universal constants $C_1$, $C_2$ such that under a
  conformal variation $(\theta^\varepsilon = e^{2\varepsilon \Upsilon
  }\theta, J_\varepsilon = J)$ we have
  $$
  (\ln\|\quad \|_{C})^\bullet = \int_M \Upsilon(C_1 \Delta_H R +
  C_2\Im A_{11,}^{\phantom{11,}11})\,\theta\wedge d\theta \,,
  $$
  where $R$, $A$, $\Delta_H$ and a comma subscript denote
  respectively the scalar curvature, torsion, sub-Laplacian and
  covariant differentiation with respect to the Tanaka--Webster
  connection (\cite{Tanaka, Webster}) of $(\theta, J)$.
\item Let $C'_1= -C_1/8$ and $C'_2= C_2/4$, with $C_1, C_2$ as above.
  Then
  $$
  \|\quad \|_{\mathrm{CR}} = \exp\left( C_1' \int_M R^2 \,
    \theta\wedge d\theta + C_2' \int_M |A|^2 \theta\wedge
    d\theta\right) \|\quad \|_{C}
  $$
  is a CR-invariant (i.e.~independent of contact form) metric on
  $\det H^*(\EE,d_H)$.
\item There exist universal constants $C_3$, $C_4$ such that both
  $$
  \|\quad \|_{H}^\nu = \exp(C_3 \nu(M)) \|\quad \|_{\mathrm{CR}}
  \quad \mathrm{and} \quad \|\quad \|_{H}^{D} = \exp(C_4
  \overline{\eta}(D*)) \|\quad \|_{\mathrm{CR}}
  $$
  are contact-invariant metrics, with $\nu(M)$ the $\nu$-invariant
  of Biquard--Herzlich~\cite{BH}, and $\overline{\eta}(D*)$ the
  CR--invariant correction to the pseudohermitian eta invariant
  $\eta(D*)$; see \cite[Theorem 9.4]{BHR}.
\end{enumerate}
\end{cor} 
\begin{proof}
  We complexify $H$ and work in a local frame $\{Z_1,
  Z_{\overline{1}}\}$ and coframe $\{\theta^1, \theta^{\overline{1}}
  \}$, where $\theta^1(T) = 0 = \theta^{\overline{1}}(T)$ ($T$ the
  Reeb field). Under a constant scaling of contact form $\that =
  K\theta$, the relevant heat coefficient $a_{t^0;\,k}$, for $k=0,1$,
  scales as $\tra(\hat{a}_{t^0;\,k}) = K^{-2}\tra( a_{t^0;\,k})$. This
  is easily verified by an argument similar to that for
  ~\cite[(6.48)]{BGS}.  Basic invariant theory (see
  e.g.~\cite{Stanton,BHR}) then tells us that $\tra(a_{t^0;\,k})$ must
  be a universal linear combination of
  \begin{equation}
    \label{eq:23}
    R^2,\ |A|^2,\ \Delta_HR,\ R_{,0} = 2 \Re A_{11,}^{\phantom{11,}11}
      \textnormal{ 
      and } \Im A_{11,}^{\phantom{11,}11}.
  \end{equation}
    
  Now $\kappa(0)$ is the integral of a linear combination of these
  terms, which is moreover independent of the choice of $\theta$. A
  familiar argument (see e.g.~\cite{BHR}) shows that the integral of a
  linear combination of $R^2$ and $|A|^2$ can never be
  contact-invariant. Thus $\kappa(0)$ is the integral of a divergence,
  and hence vanishes. This, together with \eqref{eq:21}, proves
  assertion (1).
  
  \smallskip
  
  Consider now the differential of $\ln \|\quad \|_C$ under a
  conformal change of $\theta$. This may be seen as a real $1$-form
  $\alpha$ on the space $\Theta$ of contact forms. By
  Corollary~\ref{cor:contact-quillen-var} (2) and \eqref{eq:23} it can
  be written
  \begin{displaymath}
      \alpha_\theta(\Upsilon)  
      = \int_M \Upsilon (c_1 R^2 + c_2 |A|^2 + c_3 \Delta_H R + c_4
      R_{,0} + c_5 \Im A_{11,}^{\phantom{11,}11})\, \theta \wedge d
      \theta\,,
   \end{displaymath}
   for some universal constants $c_i$. Here we identified the tangent
   space $T_\theta \Theta$ with functions $\Upsilon$ on $M$.  By
   \cite[Lemma 9.5]{BHR}, the general vanishing of such a $1$-form on
   constant $\Upsilon$ implies that $c_1= c_2 = 0$, while the fact
   that $\alpha$ is a \emph{closed} form gives $c_4 = 0$; see
   \cite{BHR} for details. This proves assertion (2).
   
   Also by (83)--(84) in \cite[\S9]{BHR}, one has
   \begin{align*}
     \frac{d}{d \Upsilon} \int_M R^2 \theta \wedge d \theta & = 8
     \int_M \Upsilon \, (\Delta_H R) \,\theta \wedge d \theta \,,
     \intertext{and} \frac{d}{d \Upsilon} \int_M |A|^2 \theta \wedge d
     \theta & = - 4 \int_M \Upsilon \, (\Im A_{11,}^{\phantom{11,}11}
     )\, \theta \wedge d \theta\,,
   \end{align*}
   leading to assertion (3).
   
   Assertion (4) is proved similarly as for the case of the contact
   eta invariant in \cite[\S9]{BHR}. The CR deformations (i.e.~of $J$)
   of the CR--invariants $\nu$, $\overline{\eta}(D*)$ and $\ln \|
   \quad \|_{\mathrm{CR}}$ are all given by multiples of
   $$
   \int_M \langle Q, \overset{\bullet}{J}\, \rangle \, \theta
   \wedge d \theta \,,
   $$
   where $Q$ is Cartan's tensor; see \cite[\S9]{BHR} for details.
\end{proof}

\begin{rem}
  As may be seen from \eqref{eq:contact-quillen-var}, in order to
  determinate the various universal constants in
  Corollary~\ref{cor:3var} and investigate whether $\|\quad \|_{C}$
  has any contact-invariant properties, one needs to calculate the
  local coefficients of $t^0$ in the diagonal small-time asymptotic
  expansion of the heat kernels of the fourth-order Laplacians we
  consider here. Formulae for calculating these coefficients are built
  into the pseudodifferential construction of the heat kernel, however
  implementing these in practice seems difficult.
\end{rem}

\section{Contact and Ray--Singer analytic torsions of CR Seifert manifolds}
\label{sec:contact-ray-singer}

We follow~\cite{BHR} to review the definition of a CR Seifert manifold
and to fix notation. Note that in dimension $3$ a calibrated almost
complex structure $J$ for the contact structure $H$ is automatically
integrable; the pair $(H,J)$ is often called a \emph{pseudoconvex CR
  structure}.

\begin{defn}
  \label{def:CR-Seifert}
  A \emph{CR Seifert manifold} is a $3$-dimensional compact manifold
  $M$ endowed with a pseudoconvex CR structure $(H, J)$ and a Seifert
  structure $\varphi: \mathbb{S}^1\to \Diff(M)$ that are compatible in
  the following sense: the circle action $\varphi$ preserves the CR
  structure and is generated by a Reeb field $T$.
\end{defn}

It is easily proved that existence of a Reeb field $T$ satisfying
$\varphi_*(d/dt) = T$ is equivalent to existence of a locally free
action of $\mathbb{S}^1$ whose (never vanishing) infinitesimal
generator preserves $(H,J)$ and is transverse everywhere to $H$.

The quotient space $\Sigma = M/\mathbb{S}^1$ is an orbifold surface
with conical singularities. Each CR Seifert manifold is then the
$\mathbb{S}^1$-bundle inside a line orbifold bundle $L$ over the
compact Riemannian orbifold $\Sigma$. Singularities of $L$ are located
above the singularities of $\Sigma$ in such a way that the total space
$M$ of the bundle is a smooth manifold: if the local fundamental group
at $p\in\Sigma$ is $\Z/\alpha\Z$ ($\alpha\in\N^*$), a generator acts
on a local chart around $p$ as $e^{i\frac{2\pi}{\alpha}}$ and on the
fibre above $p$ as $e^{i\frac{2\pi\beta}{\alpha}}$, where $\alpha$ and
$\beta$ are relatively prime integers with $1\leq\alpha<\beta$.

\begin{thm}
  \label{thm:RS}
  Let $M$ be CR Seifert manifold and $\rho:\pi_1(M)\to U(N)$ a unitary
  representation. Then:
  
  $\bullet$ The analytic torsion $T_C(\rho)$ of the twisted contact
  complex and Ray--Singer analytic torsion $T_{RS}(\rho)$ satisfy
  $$
  T_C(\rho) = \bigl(T_{RS}(\rho) \bigr)^{-1}\,.
  $$
  
  $\bullet$ The two Ray--Singer metrics on $\det H^*(M, \rho)$,
  corresponding to the de Rham and contact complexes (see
  \eqref{eq:18} and \eqref{eq:19}), coincide, i.e.
  \begin{equation*}
    \|\quad \|_C = \| \quad \|_{RS}\,,
  \end{equation*}
  via the isomorphism $ \det H^*(\mathcal{E}^*_\rho , d_H) \cong \det
  H^*(\Omega^*_\rho M, d)$ coming from
  Proposition~\ref{prop:contact-deRham}.
\end{thm}

The remainder of this section will be devoted to the proof of these
results.

We first need to compare the two spectra coming from the de Rham and
contact complexes. This has been done in \cite[\S\S7, 8]{BHR} in the
untwisted case, i.e.~for a trivial representation.  We will rely on
and refer to the spectral analysis done there and point out the few
differences coming from the use of the representation $\rho$ here.

\subsection{Circle action and Fourier analysis}
\label{sec:circle-acti-four}

Let $V$ be the flat complex vector bundle over $M$ associated to
$\rho$. It is the quotient of the trivial bundle $\widetilde M \times
\C^N$ over the universal cover $\widetilde M$ of $M$ by the deck
transformations $\gamma .(m, v)= (\tau(\gamma)m, \rho (\gamma) v)$. In
what follows the contact complex is twisted by $\rho$ in order to take
values in $V$.

Let $\varphi_t$ be the circle action on $M$ induced by the Reeb field
$T$. It may be lifted on $V$, by parallel transport for the flat
connection $\nabla_\rho$, but no longer as a circle action. We have
instead
$$
\varphi_{2\pi} = \rho(f) ,
$$
where $f = \varphi_{[0,2\pi]}(m)$ is the generic closed orbit of
the action, as seen in $\pi_1(M)$. This $f$ is central, as comes from
the presentation of the fundamental group of Seifert manifolds.
\begin{prop}[{see e.g.~\cite{FS, Nicolaescu, Scott}}]
  \label{prop:pi1}
  Let $M$ be the circle $V$-bundle $L$ of rational degree $d = b +
  \sum_i \frac{\beta_i }{\alpha_i}$ over the orbifold surface $\Sigma$
  of integral genus $g$, with $n$ conical points $x_i$ of type
  $(\alpha_i,\ \beta_i)$. Then $\pi_1(M)$ admits the presentation
  \begin{multline*}
    \pi_1(M) = \langle f,\ a_j,\ b_j\ (1\leq j\leq g),\ g_i\ (1\leq i
    \leq
    n) \mid \\
    [a_j, f]= [b_j, f] = [g_i, f] = g_i^{\alpha_i} f^{-\beta_i} =
    f^{b} \prod_j [a_j, b_j] \prod_i g_i = 1 \rangle.
  \end{multline*}
\end{prop}
We can split $V$ into irreducible components, where we have
\begin{displaymath}
  \rho(f) = e^{2i\pi x} 
\end{displaymath}
for some $x \in [0,1[$. We recover a circle action on each such
component $V^x$ by setting
\begin{displaymath}
  \psi_t= e^{-i t x}\varphi_t\,.
\end{displaymath}
Using the circle action $\psi_t$ one can still perform a Fourier
decomposition as in \cite[\S 7]{BHR}. Namely, for $f \in V^x$, one has
$$
f = \sum_{n \in \Z} \pi_n f \quad \mathrm{where}\quad \pi_n f =
\frac{1}{2\pi} \int_0^{2\pi} e^{-int }\psi_t (f) dt
$$
satisfies $\psi_t(\pi_n f) = e^{int} \pi_n f$ and $T (\pi_n f) = i
(n+x) \pi_n f$. Hence the spectrum of $iT$ becomes the shifted $\Z -
x$ on $V^x$. For $\lambda=-n-x$ we note
\begin{equation}
  \label{eq:24}
  V_\lambda = V^x_n = V^x \cap \{iT = -n -x\} = \pi_n (V^x) \,.
\end{equation}

As the circle action preserves the metric and the whole
pseudohermitian structure $(H,\theta, J)$, we can split both the
Hodge--de Rham and the contact complex spectra into their Fourier
components. This is useful for comparing the spectra.

\subsection{Comparing the Riemannian and sub-Riemannian spectra}
\label{sec:contact-spectrum}

Following Propositions 7.2 and 7.4 in \cite{BHR}, we consider the
following spaces.
\begin{defn}
  \label{defn:H_V}
  $\bullet$ Let $\mathcal{H}^2_V$ be the space of vertical $2$-forms
  $\alpha = \theta \wedge \beta$, with values in $V$, such that both
  $\alpha$ and $J\alpha$ are closed.
  
  $\bullet $ Let also $\mathcal{H}^0_V$ be the space of pluri-CR
  functions in $V$, i.e.
  $$
  \mathcal{H}^0_V = \ker (\Delta_H - T^2) = \ker
  \overline{\Box}_{V}\Box_{V}
  $$
  with $\Box_{V} = \partial_{V}^* \partial_{V}$ and
  $\overline{\Box}_{V} = \overline{\partial}_{V}^*
  \overline{\partial}_{V_\rho}$.
\end{defn}

According to \cite[(63)--(68)]{BHR} the non-zero spectrum of the
non-positive second-order Laplacian $P = D* + \delta_H d_H$ on
$2$-forms splits into
\begin{equation}
  \label{eq:25}
  \begin{split}
    \spec^* (P) & = \spec^* (D*) \ \cup \ \spec^*(\Delta_H) \\
    &  = \spec^*(\Delta_H) \ \cup \ \spec^*(- \Delta_H) \\
    & \quad \quad \bigcup \spec^* (-JT \mid \mathcal{H}^2_V) \setminus
    \spec^* (-|T|\mid \mathcal{H}^0_V) ,
  \end{split}
\end{equation}
where $\Delta_H = d_H \delta_H + \delta_H d_H$ is the horizontal
second-order Laplacian on functions.

The torsion function $\kappa_C$ of the contact complex is defined
using the fourth-order Laplacians $\Delta_H^2$ on functions, and
$\Delta_1 = D^* D + (d_H \delta_H)^2$ on $1$-forms, with $\Delta_1$
conjugated to $P^2$ by Hodge $*$ duality. Hence \eqref{eq:25} yields
\begin{equation}
  \label{eq:26}
  \spec^*(\Delta_1)  = 2 \times \spec^*(\Delta_H^2)
  \bigcup \spec^* (-T^2 \mid \mathcal{H}^2_V) \setminus \spec^*
  (-T^2\mid \mathcal{H}^0_V)\,.
\end{equation}
Finally by \eqref{eq:20} the torsion function of the contact complex
reads
\begin{equation}
  \label{eq:27}
  \begin{aligned}
    \kappa_C(s) & = 2 \zeta(\Delta_H^2)(s) -  \zeta(\Delta_1)(s) \\
    & = \zeta^*(-T^2\mid \mathcal{H}^0_V) (s) - \zeta^*(- T^2 \mid
    \mathcal{H}^2_V)(s) + \kappa(M, \rho) ,
  \end{aligned}
\end{equation}
where we have set
\begin{equation}
  \label{eq:28}
  \begin{aligned}
    \kappa(M, \rho) &= 2 \dim (\ker \Delta_H) - \dim (\ker \Delta_1) \\
    & = 2 \dim (H^0(M, \rho)) - \dim (H^1(M, \rho)) \,,
  \end{aligned}
\end{equation}
since the twisted contact complex is a resolution computing the
cohomology of $M$ with values in $V$.

\medskip

We proceed similarly for the Hodge--de Rham spectrum, and work again
with the calibrated metric $ g = d\theta(\cdot,J\cdot) + \theta^2 $.
Set
\begin{equation}
  \label{eq:29}
  Q^\pm = \pm \frac{1}{2} + \sqrt{\frac{1}{4}+ \Delta_0}
\end{equation}
where $\Delta_0$ is Hodge--de Rham Laplacian acting on functions.
According to \cite[Corollary 7.6]{BHR} the spectrum of $d*$ on
$2$-forms splits as
\begin{equation}
  \label{eq:30}
  \spec^*(d*)  = \spec^* (Q^+) \bigcup \spec^*(- Q^- \mid
  (\mathcal{H}^0_V)^\bot) \bigcup \spec^*(-JT \mid \mathcal{H}^2_V)\,.
\end{equation}
By Definition \ref{defn:H_V}, we have $\Delta_0 = \Delta_H - T^2 = |T|
- T^2$ on $\mathcal{H}^0_V$ so that
$$
Q^- = - 1/2 + \sqrt{1/4+ \Delta_0} = |T| \quad \mathrm{on}\quad
\mathcal{H}^0_V\,.
$$
Then \eqref{eq:30} also reads
\begin{equation*}
  \spec^*(d*)  = \spec^* (Q^+) \bigcup \spec^*(- Q^-) \setminus
  \spec^*(-|T|\mid \mathcal{H}^0_V) \bigcup
  \spec^*(-JT \mid \mathcal{H}^2_V)\,,
\end{equation*}
and since $\delta d$ on $1$-forms is $*$ conjugated to $(d*)^2$ on
$2$-forms, we get
\begin{equation}
  \label{eq:31}
  \spec^*(\delta d)  = \spec^*
  (Q^+)^2 \bigcup \spec^* (Q^-)^2
  \bigcup \spec^*(- T^2 \mid \mathcal{H}^2_V) \setminus
  \spec^* (-T^2 \mid \mathcal{H}^0_V) \,.
\end{equation}

Now following our convention on analytic torsion, inverse to the
original definition of Ray--Singer \cite{RS} (see
Remark~\ref{rem:R-torsion} and \eqref{eq:17}), the torsion function of
de Rham complex reads in dimension $3$ as
\begin{align*}
  \kappa_{RS} (s) & = \sum_{k=0}^3 (-1)^{k+1} k \zeta (\Delta_k)(s)
  \\
  & = 3 \zeta(\Delta_0)(s) - \zeta (\Delta_1)(s) \,,
\end{align*}
with $\Delta_i = d\delta + \delta d$ the Hodge--de Rham Laplacians on
$i$-forms. Using
$$
\zeta^*(\Delta_1)(s) = \zeta^*(\delta d)(s) +
\zeta^*(\Delta_0)(s)\,
$$
and \eqref{eq:31} one finds that
\begin{align*}
  \kappa_{RS}(s) & = 2 \zeta^*(\Delta_0)(s) - \zeta^*( \delta d)(s) +
  3 \dim (\ker \Delta_0) - \dim (\ker \Delta_1) \\
  & = 2 \zeta (\Delta_0)(s) - \zeta
  (Q^+)(2s) - \zeta(Q^-)(2s)\\
  & \quad - \zeta^*(-T^2 \mid \mathcal{H}^2_V) (s) + \zeta^*(- T^2
  \mid \mathcal{H}^0_V )(s) + \kappa(M,\rho)\,,
\end{align*}
since $\ker Q^+ = \{0\}$ and $\ker Q^- = \ker \Delta_0$ by the
definition \eqref{eq:29} of $Q^\pm$. Comparing to the contact-complex
torsion \eqref{eq:28} we have shown the following result.
\begin{prop}
  \label{prop:comparison-torsions}
  On a CR Seifert manifold, the Ray--Singer and contact complex
  torsion functions twisted by a unitary representation satisfy
  \begin{equation}
    \label{eq:32}
    \kappa_{RS}(s) - \kappa_C(s)= 2 \zeta (\Delta_0)(s) - \zeta
    (Q^+)(2s) - \zeta(Q^-)(2s)\,, 
  \end{equation}
  where $\Delta_0$ is the Hodge--de Rham Laplacian on functions and $
  Q^\pm = \pm 1/2 + \sqrt{1/4 + \Delta_0}$.
\end{prop}

Note at this stage that the right-hand side of \eqref{eq:32} vanishes
at $s=0$, as needed by the vanishing of both torsion functions at
$s=0$; see \cite{RS} and Corollary \ref{cor:contact-quillen-var}. This
also follows from $\zeta(\Delta_0)(0)= 0$, for the Hodge--de Rham
Laplacian in odd dimension, and that $\zeta(Q^+)(0) = -
\zeta(Q^-)(0)$, by \cite[Lemma 8.5]{BHR}.

\subsection{Proof of Theorem \ref{thm:RS}}
\label{sec:proof-theorem-RS}

In view of Proposition~\ref{prop:comparison-torsions} we need to show
that if we set
$$
Q(s) = \zeta(Q^+)(s) + \zeta(Q^-)(s)-2 \zeta (\Delta)(s/2),
$$
writing $\Delta$ instead of $\Delta_0$ (the Hodge--de Rham
Laplacian on functions) for brevity, then $Q'(0) = 0$. We have a hint
that this is true by examining finite energy cut-offs: at any finite
spectral level $(\Delta \leq \lambda)$ it holds that
\begin{align*}
  \zeta(Q^+)'(0) + \zeta'(Q^-)'(0) & = - \ln\det(Q^+) - \ln\det(Q^-) \\
  & = - \ln \det (Q^+ \times Q^-) \\
  &= - \ln \det \Delta\\
  & = \zeta(\Delta)'(0)\,.
\end{align*}
Hence $Q'(0)$ is insensitive to finite eigenvalues and behaves like a
pseudodifferential invariant. It may indeed be seen as a
multiplicative anomaly for the regularized determinant of the product
of the two commuting operators $Q^\pm$. As thus $Q'(0)$ is related to
a Wodzicky residue-type invariant; see \cite[\S 6.5]{Kassel}.

In fact the spectral function $Q$ makes sense on any compact
Riemannian manifold.  The following result holds in a far more general
setting than CR Seifert manifolds.
\begin{lemma}
\label{lemma:Q}
On any odd-dimensional compact Riemannian manifold $Q(0) = Q'(0) = 0$.
\end{lemma}
\begin{proof}
  Consider the one-parameter deformation
  $$
  Q_{\lambda}^\pm = \pm \lambda + \sqrt{\lambda^2 + \Delta}\,,
  $$
  so that, with $\lambda = 1/2$, $Q_{1/2}^\pm$ coincides with our
  original $Q^\pm$. Note that the product formula
  $$
  Q_\lambda^+ \times Q_\lambda^- = \Delta
  $$
  we already mentioned is preserved during the deformation.  By
  ellipticity of $Q^\pm_\lambda$ and a Mellin transform
  $$
  \zeta^*(Q^\pm_\lambda)(s) = \frac{1}{\Gamma(s)}\int_0^{+\infty}
  t^{s-1}(\tr(e^{-tQ^\pm_\lambda})-\dim\ker Q^\pm_\lambda)dt
  $$
  is holomorphic for large $s$. Define a function, holomorphic for
  large $s$,
  $$
  F^\pm(\lambda, s) = \int_0^{+\infty } t^{s-1} \tr^* (e^{-t
    Q_{\lambda}^{\pm}}) dt \,,
  $$
  where here and in the sequel $\tr^*(P) = \tr(P) -
  P(\textnormal{const.~function} =1)$. In particular
\begin{align*}
  \zeta(Q^+_{1/2})(s) = \zeta^*(Q^+_{1/2})(s) &=
  \frac{1}{\Gamma(s)}\int_0^{+\infty} t^{s-1}\tr(e^{-tQ^+_{1/2}})dt\\
  &= 1+ \Gamma(s)^{-1}F^+(1/2,s)\,,
\end{align*}
and
\begin{align*}
  \zeta(Q^-_{1/2})(s) = 1+ \zeta^*(Q^-_{1/2})(s) &= 1+
  \frac{1}{\Gamma(s)}\int_0^{+\infty} t^{s-1}\tr(e^{-tQ^-_{1/2}} - 1)dt\\
  &= 1+ \Gamma(s)^{-1}F^-(1/2,s)\,.
\end{align*}
Thus
\begin{align*}
  Q(s) & = 1 + \Gamma (s)^{-1} F^+(1/2,s) + 1 + \Gamma^{-1}(s)
  F^-(1/2, s) - 2 -2 \Gamma(s)^{-1}F(0,s) \\
  & \sim s (F^+(1/2,s) + F^-(1/2, s) - 2 F(0,s))
\end{align*}
when $s \rightarrow 0$, with $F(0,s)= F^+(0,s) = F^-(0,s) $. Therefore
we need to show that
\begin{equation*}
  F^+(1/2, 0) + F^-(1/2,0) - 2 F(0, 0) = 0,
\end{equation*}
for which it clearly suffices to show that
\begin{equation}
\label{eq:33}
  \partial_\lambda F^+ (\lambda, 0) + \partial_\lambda F^-(\lambda, 0)
  = 0.
\end{equation}

Now one has, for the smooth family of commuting elliptic first-order
operators $Q_\lambda$,
\begin{align*}
  \frac{d}{d \lambda} \bigl( e^{-t Q_{\lambda}^\pm} \bigr)& = -t
  \Bigl( \pm 1 + \frac{\lambda}{\sqrt{\lambda^2 + \Delta}} \Bigr)
  e^{-t
    Q_{\lambda}^\pm} \\
  & = \frac{\pm t}{\sqrt{\lambda^2 + \Delta}} \frac{d}{dt} \bigl(
  e^{-t Q_{\lambda}^\pm} \bigr) \,,
\end{align*}
so that
\begin{align*}
  \partial_\lambda F^+ (\lambda, s) + \partial_\lambda F^-(\lambda, s)
  & = \int_0^{+\infty} t^s \frac{d}{dt} \tr^* \Bigl(\frac{e^{-t
      Q_{\lambda}^+} - e^{-t Q_{\lambda}^-}}{\sqrt{\lambda^2 +
      \Delta}} \Bigr) dt
  \\
  \intertext{or after integrating by parts,} & = s \int_0^{+\infty}
  t^{s-1} \tr^* \Bigl(\frac{e^{-t Q_{\lambda}^+} - e^{-t
      Q_{\lambda}^-}}{\sqrt{\lambda^2 + \Delta}} \Bigr) dt
  \\
  & = s \int_0^{+\infty} t^{s-1} 2 \sinh (t\lambda) \tr^* \Bigl(
  \frac{e^{-t \sqrt{\lambda^2 + \Delta}}}{\sqrt{\lambda^2 + \Delta}}
  \Bigr) dt \,,
\end{align*}
or after again integrating by parts,
\begin{equation}
  \label{eq:34}
  \partial_\lambda F^+ (\lambda, s) + \partial_\lambda F^-(\lambda, s)
  = 2s \int_0^{+\infty} g(\lambda, t ,s) \tr^*(e^{-t \sqrt{\lambda^2 + 
      \Delta}} ) dt
\end{equation}
with
\begin{equation*}
  g(\lambda, t,s) = \int_0^t u^{s-1} \sinh (u \lambda) du \,. 
\end{equation*}

We therefore need to study the residue at $s=0$ of the integral
expression in \eqref{eq:34}. First $g$ is easily expanded as
\begin{align}
  g(\lambda, t,s) &= \int_0^t u^{s-1} \sum_{p \geq 0}
  \frac{\lambda^{2p+1} u^{2p+1}}{(2p+1)!} du \nonumber\\
  &= \sum_{p\geq 0} \frac{\lambda^{2p+1} t^{2p+1+s}}{(2p+1)!\, (2p+s
    +1)} \label{eq:35} \,.
\end{align}
Consider next the Poisson kernel $\dsp \tr^*(e^{-t \sqrt{\lambda^2 +
    \Delta}} )$ in \eqref{eq:34}; the beginning of its asymptotic
expansion as $t \searrow 0$ is related to that of the heat kernel
$\tr(e^{-t (\lambda^2 + \Delta)} )$ as follows. Recall from
e.g.~\cite[Lemma 1.7.4]{Gilkey} that the trace of the heat kernel of a
second-order elliptic Laplacian such as $P = \lambda^2 + \Delta$
develops when $t \searrow 0$ as
\begin{equation*}
  \tr(e^{-t P} ) \sim \sum_{k\geq 0} c_k t^{k - \frac{m}{2}}, 
\end{equation*}
where $m$ is the manifold dimension and the $c_k$ are integrals of
curvature terms.
\begin{prop}
  \label{prop:poisson}
  One has for $P$ and $c_k$ as above, for odd and even dimension $m$,
  \begin{equation}
    \label{eq:36}
    \tr^*(e^{-t P^{1/2} } ) = \sum_{k=0}^{[m/2]} \frac{2^{m-2k}}{\sqrt
      \pi} \Gamma\bigl( \frac{m- 2k+1}{2} \bigr) c_k t^{2 k-m} - 1 + f(t)
  \end{equation}
  where $f(t) \rightarrow 0$ when $t\searrow 0$.
\end{prop}

\begin{proof}
  This is a particular case of \cite[Theorem 3.1]{BM}. Indeed B\"ar
  and Moroianu gave there the full development of such kernels on the
  diagonal. Higher order terms in the development of the Poisson
  kernel are more involved in odd dimension since they contain log and
  even non-local coefficients.
  
  Here is an alternative proof of the partial development we need. The
  classical Laplace transform $\dsp \mathcal{L}(t^{-1/2}e^{-1/t}) =
  \sqrt\pi p^{-1/2} e^{-2 p^{1/2}}$ leads to the subordination formula
  $$
  e^{-t P^{1/2}} = \pi^{-1/2} \int_0^{+\infty} e^{-u} u^{-1/2}
  e^{-t^2 P/4u} du\,
  $$
  between Poisson and heat kernels. Therefore summing at the trace
  level,
  \begin{align*}
    \tr(e^{-t P^{1/2}}) & = \pi^{-1/2} \int_0^{+\infty} e^{-u}
    u^{-1/2}
    \tr(e^{-t^2 P/4u})  du \\
    & = \pi^{-1/2} \int_0^{+\infty} e^{-u} u^{-1/2} \bigl( \sum_{k =
      0}^{[m/2]} c_k t^{2k-m}(4u)^{m/2 - k} + B(t^2/4u )\bigr)
    du \\
    & = \sum_{k=0}^{[m/2]} \frac{2^{m-2k}}{\sqrt \pi} \Gamma\bigl(
    \frac{m- 2k+1}{2} \bigr) c_k t^{2 k-m} + \pi^{-1/2}
    \int_0^{+\infty} e^{-u} u^{-1/2} B(\frac{t^2}{4u}) du
  \end{align*}
  with $B(t^2/4u)$ bounded and $B(v) \rightarrow 0$ when $v \searrow
  0$. This gives \eqref{eq:36} by dominated convergence and the remark
  that
  $$
  \tr^*(e^{-tP^{1/2}}) = \tr (e^{-tP^{1/2}}) - e^{-t \lambda} = \tr
  (e^{-tP^{1/2}}) -1 + o(1).
  $$
\end{proof}

We can now complete the proof of Lemma \ref{lemma:Q}. We split
\eqref{eq:34} into
\begin{equation*}
  \partial_\lambda F^+ (\lambda, s) + \partial_\lambda
  F^-(\lambda, s) 
  = 2s \Bigl( \int_0^1 + \int_1^{+\infty} \Bigr) g(\lambda, t ,s)
  \tr^*(e^{-t P^{1/2} 
  }) dt \,.
\end{equation*}
By \eqref{eq:35} the second integral here is meromorphic with simple
poles at $s = -2n -1$ for $n \in\N$; in particular it is regular at
$s=0$. Set
$$
c'_k=\frac{2^{m-2k}}{\sqrt \pi} \Gamma\Bigl(\frac{m- 2k+1}{2}\Bigr)
c_k \,,
$$
so that by \eqref{eq:36}
\begin{align*}
  \int_0^1 g(\lambda, u, s) \tr^*(e^{-uP^{1/2}}) du & = \int_0^1
  \sum_{0\leq k\leq [m/2]} \sum_{p \geq 0} \frac{\lambda^{2p+1}
    u^{2p+1+s}}{(2p+1)! (2p+1+s)} c'_k u^{2k-m} du \\
  & \quad + \int_0^1 g(\lambda, u, s) (f(u)-1) du \\
  & = \sum_{0\leq k\leq [m/2]} \sum_{p \geq 0} \frac{c'_k
    \lambda^{2p+1}}{(2p+1)! (2p+1+s)(2p+2+2k-m +s)} \\
  & \quad \quad + \textnormal{holomorphic terms for $\Re(s)>-1$}\,.
 \end{align*}
 This expression has no pole at $s=0$ if $m$ is odd, giving
 \eqref{eq:33} and hence Lemma~\ref{lemma:Q}.
\end{proof}

\smallskip

We now prove Theorem~\ref{thm:RS}. First
Proposition~\ref{prop:comparison-torsions}, Lemma~\ref{lemma:Q} and
\eqref{eq:17} show that
$$
T_C(\rho) = \exp(\kappa'_C(0)/2) = \exp(\kappa'_{RS}(0)/2) =
(T_{RS}(\rho))^{-1}\,.
$$

The equality of Ray--Singer metrics now comes from \eqref{eq:18} and
\eqref{eq:19}, using the equality of $L^2$ metrics on $H^*(M,\rho)$
when $H^*(M,\rho)$ is represented by harmonic forms in the de Rham and
contact complexes. Indeed these latter two notions coincide on CR
Seifert manifolds because of vanishing Tanaka--Webster torsion; see
\cite[Proposition 12]{Rumin94}.

\section{The torsion function of CR Seifert manifolds and its
  dynamical aspects}
\label{sec:torsion-function-seifert}

As an illustration of our viewpoint on analytic torsion, we first show
how to compute it on any CR Seifert $3$-manifold $M$ equipped with a
unitary representation $\rho : \pi_1(M) \rightarrow U(N)$.

\smallskip

As we will be only concerned with the contact torsion function
$\kappa_C$ in the sequel, we will denote it by $\kappa$ for brevity.

Surprisingly, on CR Seifert manifolds, we will see that the
\emph{whole} contact torsion function $\kappa$, not only $\kappa'(0)$,
is expressible using topological data and combinations of
Riemann--Hurwitz zeta functions, parametrised by dynamical properties
of $\rho$ with respect to the circle action on $M$. This leads in
particular to a Lefschetz-type formula for the torsion; see
Theorem~\ref{thm:Lefschetz-torsion}. This extends a result obtained by
Fried \cite{Fried} in the acyclic case using topological methods.

In fact it turns out that our spectral torsion function $\kappa(s)$
may also be seen as a purely dynamical zeta function, constructed from
holonomies along all closed orbits of the Reeb field $T$ and its
length spectrum. This will be shown in \S\S
\ref{sec:dynam-aspects-cont} and \ref{sec:heat-kernel-as}.

\subsection{Torsion function $\kappa$ and the Riemann--Roch--Kawasaki
  formula}
\label{sec:torsion-riemann-roch}

We first relate $\kappa$ to basic holomorphic data.

\smallskip

Let $\overline{V}$ be the conjugate complex vector space to $V$,
i.e.~the same underlying real space with the opposite complex
structure, and set
$$
W = V \oplus \overline{V} \,.
$$
Consider the $\overline{\partial}_W$ complex
\begin{equation}
  \label{eq:37}
  \overline{\partial}_W : \Omega^0 W = \Gamma(M, W) \longrightarrow
  \Omega^{0,1}  W = \Omega^{0,1} H \otimes W = \Gamma(M, \Lambda^{0,1}
  H^* \otimes W )
\end{equation}
given on the CR (holomorphic) bundle $W$ and let
\begin{equation}
  \label{eq:38}
  \mathcal{H}^0_W = \ker \overline{\partial}_W \quad \mathrm{and} \quad
  \mathcal{H}^1_W = \ker \overline{\partial}_W^*
\end{equation}
denote its cohomology. Using $\overline{\Box}_{V_\rho} - \Box_{V_\rho}
= iT$, see e.g.~\cite[(57)]{BHR}, one gets
\begin{equation*}
  \spec (-|T|\mid \mathcal{H}^0_V)= \spec (iT \mid \mathcal{H}^0_W) \,,
\end{equation*}
while, using Hodge $*$ duality,
\begin{equation*}
  \spec (-JT \mid \mathcal{H}^2_V) = \spec (iT \mid \mathcal{H}^1_W).
\end{equation*}
Then the spectral decomposition \eqref{eq:25} reads as follows.
\begin{prop}
  \label{prop:spec_P} The spectrum of $P= D* + \delta_H d_H$ splits
  as
  \begin{displaymath}
    \begin{split}
      \spec^* (P) & = \spec^*(\Delta_H) \ \cup \ \spec^*(- \Delta_H) \\
      & \quad \quad \bigcup \ \spec^*(iT \mid \mathcal{H}^1_W)
      \setminus \spec^* (iT \mid \mathcal{H}^0_W) \,,
    \end{split}
  \end{displaymath}
  where $W = V \oplus \overline{V}$ and $\mathcal{H}^*_W$ is the
  cohomology of $\overline{\partial}_W$ as in
  \eqref{eq:37}--\eqref{eq:38}.
\end{prop}
\begin{rem}
  Compared to the trivial representation case treated in
  \cite[(68)]{BHR}, the only change here is the tensorisation by $W$.
\end{rem}
Now using the $\overline{\partial}_W$ complex, equation \eqref{eq:26}
yields
\begin{equation*}
  \spec^*(\Delta_1)  = 2 \times \spec^*(\Delta_H^2)
  \bigcup \spec^* (-T^2 \mid \mathcal{H}^1_W) \setminus \spec^*
  (-T^2\mid \mathcal{H}^0_W)\,,
\end{equation*}  
and \eqref{eq:27} becomes
\begin{equation}
  \label{eq:39}
  \begin{aligned}
    \kappa(s) & = 2 \zeta(\Delta_H^2)(s) - \zeta(\Delta_1)(s) \\
    & = \zeta^*(-T^2 \mid \mathcal{H}^1_W)(s) - \zeta^*(-T^2 \mid
    \mathcal{H}^0_W)(s) + \kappa(M, \rho)\,.
  \end{aligned}
\end{equation}
This Lefschetz-type formula for $\kappa$ can be seen more
topologically. Indeed, Fourier decompose each $V^x$ and let
$$
W = \bigoplus_{\lambda \in \spec(iT)} W_\lambda \quad \mathrm{with
}\quad W_\lambda = V_{\lambda } \oplus \overline{V_{ \lambda}}\,.
$$
Then using the holomorphic genus
$$
\chi_{\overline{\partial}}(W_{\lambda}) = \dim
\mathcal{H}^1_{W_\lambda} - \dim \mathcal{H}^0_{W_\lambda},
$$
the torsion function also reads as the Dirichlet series
\begin{equation}
  \label{eq:40}
  \kappa (s) = \sum_{\lambda \in \spec^* (iT)}
  \frac{\chi_{\overline{\partial}} (W_\lambda)}{\lambda^{2s}} +
  \kappa(M, \rho),
\end{equation}
where, from \eqref{eq:24}, $\spec (iT)$ splits into copies of $-(\Z
+x)$ on each sub-representation $V^x$ of $V$, on which $\rho(f) =
e^{2i\pi x}$.
\begin{rem}
  \label{rem:eta-torsion}
  For comparison, the eta function of $P = D* + \delta_H d_H$ twisted
  by $\rho$ may be expressed using Proposition \ref{prop:spec_P} in a
  similar manner. One gets
  \begin{equation}
    \label{eq:41}
    \eta_\rho(P) (s) = \sum_{\lambda \in \spec^* (iT)} \mathrm{sign}(\lambda)
    \frac{ \chi_{\overline{\partial}}
      (V_{\lambda}) - \chi_{\overline{\partial}}
      (\overline{V_{\lambda}})}{|\lambda|^s} \,, 
  \end{equation}
  which is strikingly the `odd version' of the formula \eqref{eq:40}
  for the torsion function $\kappa$. Note that by \cite[Theorem
  8.8]{BHR}, $\eta_\rho (P)(0)$ identifies with $\eta_0(M, \rho)$, the
  diabatic limit of the Riemannian eta invariant with value in $\rho$,
  i.e., the constant term in the development of $\eta(M,
  g_\varepsilon, \rho)$ for the diabatic metrics (which we also
  consider in the present paper) $g_\varepsilon = \varepsilon^{-1 } d
  \theta + \varepsilon^{-2} \theta^2$ .
\end{rem}

In order to express the series \eqref{eq:40} using the
Riemann--Roch--Kawasaki formula, we need first to see the bundles
$V_\lambda$, a priori defined over $M$, as $V$-bundles over the
orbifold $\Sigma$, and compute their degrees and orbifold exponents.

\smallskip

Recall that we consider in \eqref{eq:24} the circle action given by
\begin{displaymath}
  \psi_t = e^{-itx}\varphi_t= e^{int}\quad \mathrm{on} \quad
  V_\lambda= V^x_n \,,
\end{displaymath} 
where $\varphi_t$ acts by parallel transport along $T$ in the flat
$V$. Now $e^{it}$ is a local coordinate on the fibre of the circle
$V$-bundle $M = S(L) \rightarrow \Sigma$, hence $e^{int}$ is a
coordinate for $S(L^n)$. Thus we have
\begin{equation}
  \label{eq:42}
  V_{\lambda} =  S(L^{-n}) \otimes V_0^x\,.
\end{equation}
Sections of $V^x_0$ are invariant by the circle action $\psi_t$ and
satisfy $\varphi_t = e^{itx}$. Then one can see $V^x_0$ as defined
over $\Sigma$, but equipped with the non-flat connection
$$
\nabla_x^\Sigma = \nabla_{\rho}^\mathrm{flat} - ix \theta\,,
$$
for which sections $s$ of $V_0^x$ are parallel along $T$.

Therefore, as seen from $\Sigma$, $(V_0^x, \nabla^\Sigma)$ has
curvature $ \Omega = -i x d \theta $ and rational degree (see
\cite{BHR, FS, Nicolaescu})
\begin{align*}
  d(V_0^x) & = \frac{i}{2\pi} \int_\Sigma \tr_{V_0^x}(- i x d
  \theta)   \\
  & = - \dim( V^x) x d(L)\,,
\end{align*}
because $\Omega_L= i d \theta$. Recall that by
Proposition~\ref{prop:pi1}, $d(L) = b + \sum_i
\frac{\beta_i}{\alpha_i}$. Hence by \eqref{eq:42} we get
\begin{equation}
  \label{eq:43}
  d(V_\lambda) = \dim (V^x) (-n d(L) -x d(L)) = \dim (V^x)
  \lambda d(L).
\end{equation}
By conjugation we have also
\begin{displaymath}
  d(\overline{V_{\lambda}}) =  - \dim (V^x)
  \lambda d(L) =  - d(V_{\lambda})\,,
\end{displaymath}
so that finally
$$
d (W_\lambda) = d(\overline{V_{\rho, \lambda}}) + d(V_{\rho,
  \lambda}) = 0 \,,
$$
as expected for this smooth part due to the real structure on
$W_\lambda= V_\lambda \oplus \overline{V_\lambda}$; see \cite[\S
14]{Milnor}.  We did the above computation for completeness, as the
degree of $V_\lambda$ is needed to study the eta function given in
Remark~\ref{rem:eta-torsion}.

\medskip

We determine now the local action on $V_\lambda$ of $\Gamma_i = \Z/
\alpha_i \Z$ around an orbifold point $x_i \in \Sigma$. Recall that
locally over $x_i$, $M= S(L)$ is the quotient of $\C \times S^1$ by
the action of $\Gamma_i$ generated by
$$
g . (z_1, z_2) = (e^{2i\pi / \alpha_i} z_1, e^{2i\pi \beta_i/
  \alpha_i} z_2)\,.
$$
Since $z_2= e^{it}$ here, we get the local action on $V_0^x$
\begin{align*}
  g . (z_1, z_V) = (e^{2i\pi / \alpha_i} z_1, \psi( 2 \pi\beta_i/
  \alpha_i) z_V)\,.
\end{align*}
Let $f_i = \varphi([0, 2\pi/ \alpha_i])$ be the exceptional closed
primitive orbit over $x_i$. We have
\begin{align}
  \psi(2 \pi \beta_i/ \alpha_i) & = e^{-2i \pi \beta_i x/\alpha_i}
  \varphi(2 \pi \beta_i /\alpha_i)  \nonumber \\
  & = e^{-i 2\pi \beta_i x/\alpha_i} \rho(f_i)^{\beta_i}.
  \label{eq:44}
\end{align}
As $f_i^{\alpha_i} =f$ in $\pi_1(M)$, $\rho(f_i)^{\alpha_i} = \rho(f)
$ and the spectra of $\rho(f_i)$ satisfy
\begin{equation}
  \label{eq:45}
  \left\{
    \begin{gathered}
      \spec \rho(f_i) = \{e^{2 i\pi x_{i,j}} \mid 1 \leq j \leq \dim
      V_0^x \} \\
      \mathrm{with}\ x_{i,j} = \frac{x + k_{i,j}}{\alpha_i} \in [0,1)\ 
      \mathrm{and}\ k_{i,j} \in \Z \,.
    \end{gathered}
  \right.
\end{equation}
Hence \eqref{eq:44} means that $\mathrm{spec}(g) = \{e^{2i\pi k_{i,j}
  \beta_i/\alpha_i}\}$ on the fibres of $V_0^x$ around $x_i$, and, by
tensorisation with $L^{-n}$ in \eqref{eq:42}, the isotropy exponents
of the action on $V_\lambda$ are all the couples $(\alpha_i, (- n +
k_{i,j}) \beta_i \mod \alpha_i)$.

\medskip

Now the Riemann--Roch--Kawasaki formula (see \cite{FS,Nicolaescu,BHR})
states that
\begin{align}
  \chi_{\overline\partial} (W_\lambda) & = \dim (W_\lambda)(1-g) +
  \mathrm{deg}(W_\lambda) - \sum_{i,j}
  \Bigl\{\frac{\beta_i(W_\lambda)}{\alpha_i(W_\lambda)}\Bigr\} \nonumber \\
  & = \dim (V^x) \chi(\widetilde \Sigma) - \sum_{i,j} \biggl\{\frac{(-
    n + k_{i,j}) \beta_i}{\alpha_i} \biggr\} + \biggl\{\frac{(n -
    k_{i,j}) \beta_i}{\alpha_i} \biggr\}\, ,
  \label{eq:46}
\end{align}
where $\{a\} = a - [a] \in [0,1)$ is the fractional part of $a$ and
$\chi (\widetilde \Sigma) = 2 -2g$ is the Euler characteristic of the
smooth surface $\widetilde \Sigma$ associated to the orbifold
$\Sigma$; see e.g.~\cite{Nicolaescu}.  Observe that \eqref{eq:46} does
give integers since $\{a\}+ \{-a\}$ is $0$ when $a\in\Z$ and $1$
otherwise. Recall also that, to ensure smoothness of the $V$-bundle
$M= S(L)$, the numbers $\alpha_i$ and $\beta_i$ are assumed relatively
prime, thus giving a free action of $\Z / \alpha_i \Z$ at orbifold
points.  Hence the fractional part in \eqref{eq:46} simplifies using
\begin{equation}
  \label{eq:47}
  \delta(n,i,j) =
  \begin{cases}
    1\quad & \mathrm{if} \quad n - k_{i,j} \in \alpha_i \Z \\
    0 \quad & \mathrm{otherwise} \,.
  \end{cases}
\end{equation}
Then we have
\begin{equation}
  \label{eq:48}
  \chi_{\overline\partial} (W_\lambda) = \dim (V^x)
  \chi(\Sigma^*) + \sum_{i,j} \delta(n,i,j) \,,
\end{equation}
where
$$
\chi(\Sigma^*) = 2-2g - |I| \,
$$
is the Euler characteristic of the punctured surface $\dsp \Sigma^*
= \Sigma \setminus \cup_I \{x_i\}$ at the $|I|$ orbifold points.

\medskip

For $a \in ]0,1]$ let $\dsp \zeta(s,a) = \sum_{n \in \N}
\frac{1}{(n+a)^s}$ be Hurwitz zeta function. We can now express
$\kappa$ as a combination of such functions. This is the first step
towards the identification of the torsion function as a dynamical zeta
function given in \S\ref{sec:dynam-aspects-cont}.
\begin{thm}
  \label{thm:kappa_x}
  Split $V$ into irreducible $V^x$; then the torsion function
  spectrally decomposes as
  \begin{displaymath}
    \kappa(s) = \sum_{V^x} \kappa_x (s)
  \end{displaymath}
  such that :
  
  $\bullet$ On $V^x$ with $x \in \, ]0,1[$, i.e.~$\rho(f) = e^{2i\pi
    x} \not= \mathrm{Id}$, we have
  \begin{equation}
    \label{eq:49}
    \begin{split}
      \kappa_x (s) = \dim (V^x) \chi(\Sigma^*) \bigl(\zeta(2s,x)
      +\zeta(2s, 1-x) \bigr) \\
      + \sum_{i,j} \frac{1}{\alpha_i^{2s}} \bigl(\zeta(2s, x_{i,j}) +
      \zeta(2s, 1-x_{i,j}) \bigr) \,.
    \end{split}
  \end{equation}
  
  $\bullet$ On $V^0 = \ker (\mathrm{Id} - \rho(f))$ let $V^{0,i} =
  \ker (\mathrm{Id} - \rho(f_i))$; then we have
  \begin{equation}
    \label{eq:50}
    \begin{split}
      \kappa_0 (s) = \kappa(M, \rho) (2 \zeta(2s) + 1) + 2\zeta (2s)
      \sum_i \dim (V^{0,i})\bigl(\alpha_i^{-2s } -1 \bigr)  \\
      + \sum_{i,j \, |\, x_{i,j} \not=0} \frac{1}{\alpha_i^{2s}}
      \bigl(\zeta(2s, x_{i,j}) + \zeta(2s, 1-x_{i,j}) \bigr) \,.
    \end{split}
  \end{equation}
\end{thm}
This relates the torsion function to dynamical properties of the
circle action here.  Indeed apart from the cohomological term
$\kappa(M,\rho)$, the expression is clearly built on the holonomy
properties of $\rho$ along the various closed primitive orbits of the
flow: the generic orbit $f$ of the action over $\Sigma^*$, and
associated holonomy $\rho(f) = e^{2i \pi x}$ on $V^x$, and the
exceptional orbits $f_i$ of holonomy $\rho(f_i)= \{e^{2 i \pi
  x_{i,j}}\}$.

\begin{proof}

  $\bullet$ We compute first the contribution of $V^x$ for $x\not=0$,
  i.e.~when $\rho(f) = e^{2 i\pi x} \not= \mathrm{Id}$.  Here $iT =
  \lambda = -n-x \not=0$ always, and by \eqref{eq:40} and
  \eqref{eq:48}
  \begin{align*}
    \kappa_x (s) & = \dim (V^x) \chi(\Sigma^*) \sum_{n \in \Z}
    \frac{1}{|n+x|^{2s}} + \sum_{i,j} \sum_{k \in \Z} \frac{1}{|k
      \alpha_i + k_{i,j} + x|^{2s}} \\
    & = \dim (V^x) \chi(\Sigma^*) \Bigl(\sum_{n \geq 0}
    \frac{1}{|n+x|^{2s}} + \sum_{n > 0
    } \frac{1}{|-n+x|^{2s}}\Bigr) \\
    & \quad \quad + \sum_{i,j} \frac{1}{\alpha_i^{2s}} \sum_{k \in \Z}
    \frac{1}{|k + x_{i,j}|^{2s}},
  \end{align*}
  by \eqref{eq:45}. This leads to \eqref{eq:49}.

  \smallskip
  
  $\bullet$ We compute now $\kappa_0$, including the cohomological
  term $\kappa(M, \rho)$ from \eqref{eq:28} and \eqref{eq:40}, since
  harmonic forms only appear in $\ker (iT) \subset V_{\rho,0}$; see
  e.g.~\cite[Proposition 12]{Rumin94}. We have
  \begin{equation}
    \label{eq:51}
    \kappa_0(s)  = \dim (V^0)\chi(\Sigma^*)  \sum_{n \in
      \Z^* 
    }\frac{1}{|n|^{2s}} 
    + \sum_{i,j} \sum_{n \in\Z^*} \frac{\delta(n,i,j)}{
      |n|^{2s}} + \kappa (M, \rho) \, .
  \end{equation}
  We recall from \eqref{eq:28} that
  $$
  \kappa (M, \rho) = 2 \dim H^0(M, \rho) - \dim H^1(M,\rho)
  $$
  can be computed using contact-harmonic forms on $M$. By
  \cite[Proposition 12]{Rumin94} contact-harmonic forms are both
  holomorphic and $T$-invariant since the Reeb flow preserves $J$
  here, i.e.~Tanaka--Webster torsion vanishes.  Then one gets
  \begin{align}
    \kappa(M, \rho) &= 2 \chi_{\overline \partial} (W_0) \nonumber \\
    & = \dim (V^0) \chi(\Sigma^*) + \sum_{i,j} \delta(0,i,j) \quad
    \mathrm{by}\ \eqref{eq:48}  \nonumber \\
    \kappa(M, \rho) & = \dim (V^0) \chi(\Sigma^*) + \sum_{i} \dim
    (V^{0,i}) \,,
    \label{eq:52}
  \end{align}
  since, by \eqref{eq:45} and \eqref{eq:47}, $\delta(n,i,j)= 1$ if and
  only if $x_{i,j} = 0$.  Then \eqref{eq:51} reads
  \begin{equation*}
    \kappa_0(s) = \kappa(M, \rho) (2 \zeta(2s) + 1)
    + \sum_{i,j} \sum_{n \in\Z^*} \frac{\delta(n,i,j)}{
      |n|^{2s}} - 2 \zeta(2s) \sum_{i} \dim (V^{0,i})\,.
  \end{equation*}
  We observe now that if $x_{i,j} \in \, ]0,1[$,
  \begin{align*}
    \sum_{n \in\Z^*} \frac{\delta(n,i,j)}{ |n|^{2s}} & = \sum_{k \in
      \Z} \frac{1}{|k\alpha_i + k_{i,j}|^2s} \\
    & = \frac{1}{\alpha_i^{2s}} \bigl(\zeta(2s, x_{i,j}) + \zeta(2s,
    1- x_{i,j}) \bigr) \quad \mathrm{by}\ \eqref{eq:45}.
  \end{align*}
  On the other hand for $x_{i,j}=0$,
  $$
  \sum_{n \in \Z^*} \frac{\delta(n,i,j)}{|n|^{2s}} = \sum_{k \in
    \Z^*} \frac{1}{|k\alpha_i|^{2s}} = \frac{2}{\alpha_i^{2s}} \zeta
  (2s)\,,
  $$
  as needed in \eqref{eq:50}.
\end{proof}

The expression of $\kappa_x$ given in Theorem \ref{thm:kappa_x}
vanishes at $s=0$, as it ought to by Corollary \ref{cor:3var}. Here
this follows from the classical result $ \zeta(0,a) = 1/2 -a $; see
\cite[\S 13]{WW} for instance. The following observation will be
useful in the sequel.
\begin{cor}
  \label{cor:residue_kappa}
  The torsion function $\kappa$ has a unique simple pole at $s=1/2$
  with residue
  $$
  \mathrm{Res}_{1/2}(\kappa) = \chi(\Sigma) \dim V ,
  $$
  where $\chi(\Sigma) = 2-2g + \sum_i (\frac{1}{\alpha_i} -1)$
  denotes the rational Euler class of the orbifold $\Sigma$.
\end{cor}
\begin{proof}
  One knows (see \cite[\S 13]{WW}) that the Hurwitz zeta function
  $\zeta(2s,a)$ has a unique simple pole at $s=1/2$ with residue
  $1/2$.  Then \eqref{eq:49} on $V^x$ yields
  \begin{displaymath}
    \mathrm{Res}_{1/2}(\kappa_x)  = \dim V^x \chi(\Sigma^*) + \dim
    V^x \sum_i \frac{1}{\alpha_i} = \dim V^x \chi(\Sigma)\,.
  \end{displaymath}
  On the other hand by \eqref{eq:50} in $V^0$,
  \begin{align*}
    \mathrm{Res}_{1/2}(\kappa_0) & = \kappa(M,\rho) + \sum_i \dim
    (V^{0,i}) (\frac{1}{\alpha_i} - 1) + \sum_i \dim((V^{0,i})^\bot)
    \frac{1}{\alpha_i} \\
    & = (\kappa(M, \rho) - \sum_i \dim V^{0,i}) + \dim V^0 \sum_i
    \frac{1}{\alpha_i} \\
    & = \dim V^0 \bigl(\chi(\Sigma^*) + \sum_i \frac{1}{\alpha_i}
    \bigr) = \dim V^0 \chi(\Sigma) \,,
  \end{align*}
  by \eqref{eq:52}.
\end{proof}

\medskip

\begin{rem}
  We lastly observe that a similar treatment applies to handle the
  twisted eta function in Remark~\ref{rem:eta-torsion}.  Indeed by the
  Riemann--Roch--Kawasaki formula and \eqref{eq:43} one has
  $$
  \chi_{\overline \partial} (V_\lambda) - \chi_{\overline \partial}
  (\overline{V_\lambda}) = 2 \dim(V_\rho^x) \lambda d(L) + \sum_{i,j}
  \biggl\{\frac{(- n + k_{i,j}) \beta_i}{\alpha_i} \biggr\} -
  \biggl\{\frac{(n - k_{i,j}) \beta_i}{\alpha_i} \biggr\}\, ,
  $$
  and by \eqref{eq:41} the contribution of $V^x$ to eta is
  \begin{align*}
    \eta_x(P)(s) & = 2 \dim(V^x) d(L) \sum_{\lambda \in
      \spec^*(iT)} \frac{1}{|\lambda|^{s-1}} \\
    & \quad + \sum_{i,j} \sum_{\lambda \in \spec^*{iT}} \Bigl(2
    \biggl\{\frac{(- n + k_{i,j}) \beta_i}{\alpha_i} \biggr\} - 1 +
    \delta(n,i,j) \Bigr) \frac{\mathrm{sgn }(\lambda)}{
      |\lambda|^s}\,.
  \end{align*}
  The generic smooth contribution may be written as
  \begin{equation*}
    2 \dim(V^x) d(L) \times
    \begin{cases}
      2 \zeta(s-1) & \mathrm{if}\quad x=0 \\
      \zeta(s-1, x)+ \zeta(s-1, 1- x) &\mathrm{if}\quad x\not= 0\,,
    \end{cases}
  \end{equation*}
  taking value
  \begin{equation*}
    - \dim(V^x) d(L) \bigl( \frac{1}{6} + x(1-x) \bigr)
  \end{equation*}
  at $s=0$; see \cite[\S 13]{WW}.  Following Nicolaescu's work
  \cite{Nicolaescu}, the remaining `periodic' eta term can be handled
  using Dedekind--Rademacher sums and Hurwitz functions; see
  Proposition~1.10 and Lemma~1.11 in \cite{Nicolaescu} for details.
\end{rem}

\subsection{A Lefschetz-type formula for the Ray--Singer metric}
\label{sec:lefsch-type-form}

Using Theorem~\ref{thm:RS} we can now compute the Ray--Singer analytic
torsion $T_{RS} = \exp(- \kappa'(0)/2) = (T_{C})^{-1}$, which gives
the associated Ray--Singer metric on $\det H^*(M, \rho)$,
\begin{equation}
  \label{eq:53}
  \| \quad \|_{RS} = (T_{RS})^{-1}\, |\quad |_{L^2(\Omega^*M)}\,.
\end{equation}


\smallskip

In the acyclic case, i.e.~$H^*(M, \rho) = 0$, Fried~\cite{Fried} has
shown that the Reidemeister--Franz torsion, and thus the analytic
torsion by the works \cite{Cheeger, Muller} of Cheeger and M\"uller,
may be nicely expressed `à la Lefschetz' using determinants associated
to the generic and exceptional holonomies along the primitive orbits
of the circle action. For a general unitary representation $\rho :
\pi_1(M) \rightarrow U(N)$ we obtain:
\begin{thm}
  \label{thm:Lefschetz-torsion}
  Let $\rho(f)^{\top}$ and $\rho(f_i)^{\top}$ denote the restriction
  of these holonomies to respectively $(V^0)^{\bot}$ and
  $(V^{0,i})^\bot$ with
  $$
  V^0 = \ker (\mathrm{Id} - \rho(f)) \quad \mathrm{and}\quad
  V^{0,i} = \ker (\mathrm{Id} - \rho(f_i)) \,.
  $$
  Then Ray--Singer analytic torsion $T_{RS}(M, \rho)=
  \exp(-\kappa'(0)/2)$ is given by
  \begin{equation}
    \label{eq:54}
    T_{RS}(M, \rho) = (2 \pi)^{\kappa(M, \rho)}
    |\mathrm{det}(\mathrm{Id} - \rho(f)^{\top})|^{\chi(\Sigma^*)} 
    \prod_i \frac{|\mathrm{det}(\mathrm{Id} - \rho(f_i)^{\top})|}{
      \alpha_i^{\dim (V^{0,i})}} \,.
  \end{equation} 
\end{thm}

\begin{proof}
  By Lerch's formula $\partial_s \zeta(s, x)_{s=0} = \ln \Gamma(x) -
  \frac{1}{2} \ln (2\pi)$, see \cite[\S 13]{WW}, we have
  \begin{align*}
    \partial_s \zeta(0, x) + \partial_s \zeta(0, 1-x) & = \ln \bigl(
    \Gamma(x)\Gamma(1-x)/2\pi \bigr) \\
    & = - \ln (2 \sin(\pi x)) \quad \mathrm{by\ the\ Euler\ 
      reflection\ 
      formula,}\\
    & = - \ln | 1 - e^{2i \pi x}| \,.
  \end{align*}
  Hence by \eqref{eq:49} on $V^x$,
  \begin{equation*}
    -\kappa'_x(0) /2  = \dim(V^x) \chi(\Sigma^*) \ln |1- e^{2i
      \pi x}| + \sum_{i,j} \ln |1 - e^{2i\pi x_{i,j}}|\,,
  \end{equation*}
  which gives the determinant contribution of $V^x$ to \eqref{eq:54}.
  
  By \eqref{eq:50} on $V^0$ and Lerch's formula again, one finds
  \begin{equation*}
    - \kappa'_0(0)/2  = -2\zeta'(0) \kappa(M,\rho) - \sum_i \dim
    (V^{0,i}) \ln
    \alpha_i
     + \sum_{i,j \, | \, x_{i,j} \not=0} \ln |1 - e^{2i\pi
      x_{i,j}}|
  \end{equation*}
  similarly as above. This gives the needed contribution of $V^0$ to
  \eqref{eq:54}.
\end{proof}

As required, formula \eqref{eq:54} coincides with that of Fried
\cite[p.~198]{Fried} for acyclic representations. The only new factor
in our case is the cohomological term $ (2\pi)^{\kappa(M,\rho)} $.
That the full expression for the torsion is `quantised' here is due to
the rigidity of volume in this CR Seifert case. Namely, the size of
$\theta$ is fixed such that the circle action is generated by the Reeb
field $T$ in constant time $2\pi$, hence the volume forms $
\mathrm{dvol} = \theta \wedge d \theta$ on $TM$ and $d \theta$ on $H$
are also fixed. Then when the complex structure $J$ changes on the
base $\Sigma$, together with the calibrated metric $g = d\theta
(\cdot, J \cdot) + \theta^2$, harmonic forms representing $H^*(M,
\rho)$ change, but their length in $\det \mathcal{H}^*(M, \rho)$ does
not. Hence the $L^2$ metric used in \eqref{eq:53} is independent of
$J$, as is the analytic torsion given in \eqref{eq:54}.

\subsection{The contact torsion function as a dynamical zeta function}
\label{sec:dynam-aspects-cont}

In \cite{Fried87, Fried}, Fried proposed to express the torsion using
the following basic dynamical objects.

For each free homotopical class $C$ of periodic orbit of the Reeb
field $T$, let $\ell(C)$ denote its length and $\ind(C)$ its Fuller
index; see \cite{Fuller} or \cite[\S 4]{Fried87} for an account of
these notions.
\begin{prop}[{\cite[Lemma 5.3]{Fried87}, \cite{Fried}}]
  \label{prop:ind}
  The free homotopy classes of closed orbits of $T$ are the following
  :
  \begin{enumerate}
  \item $f^n$ with $n \in \N^*$, of length $2\pi n$ and Fuller index
    $\chi(\Sigma)/n$, where
    $$
    \chi(\Sigma) = 2-2g - \sum_i ( 1- 1/\alpha_i) = \chi(\Sigma^*)
    + \sum_i 1/\alpha_i \,
    $$
    is the rational Euler class of the quotient orbifold $\Sigma =
    M/\langle T \rangle$;
  \item the isolated $f_i^n$ for $n \notin \alpha_i \N$, of length
    $2\pi n /\alpha_i$ and Fuller index $1/ n$.
  \end{enumerate}
\end{prop}
Fried observed that for acyclic unitary representations one has
\begin{equation}
  \label{eq:55}  
  T_{RS}(M, \rho) = \bigl| \exp ( Z_F(0)) \bigr|\,,
\end{equation}
where $Z_F(0)$ stands for the analytic continuation at $s=0$ of the
dynamical function
\begin{equation*}
  Z_F(s) =  - \sum_{C} \ind(C) \tr(\rho(C)) e^{-s \ell (C)}\,.
\end{equation*}
This can be checked directly from the calculation of torsion in
Theorem \ref{thm:Lefschetz-torsion} and Proposition \ref{prop:ind}, as
in \cite[\S 1]{Fried}. Such a link between analytic torsion and flow
dynamics is not coincidental; it has already been observed in many
other geometric situations, see e.g.~\cite{Fried87}. In particular it
holds for the geodesic flow on hyperbolic manifolds, as proved by
Fried in \cite{Fried88}, or more generally on locally symmetric spaces
of non-positive sectional curvature, as proved by Moscovici and
Stanton in \cite{MS}.  These results are rooted in Selberg's trace
formula, expressing heat kernel traces as a sum of traces along closed
geodesics.

In our setting, in view of Theorem \ref{thm:kappa_x}, it is quite
natural to try to express the contact torsion function $\kappa(s)$
itself using the same dynamical data as above.  Indeed the
\emph{whole} spectral function $\kappa$ may be nicely interpreted `à
la Selberg' as a purely dynamical zeta function of the Reeb flow.
\begin{thm}
  \label{thm:torsion-zeta}
  Let
  \begin{displaymath}
    f(s) = \Gamma(s) \cos(\frac{\pi s}{2})\,.
  \end{displaymath}
  For $C$ closed let $\rtr(\rho(C))$ denote the \textrm{real} part of
  the trace of $\rho(C)$ on $V_\rho$. Then
  \begin{equation}
    \label{eq:56}
    f(s) \bigl( \kappa(s/2) - \kappa(M, \rho) \bigr) = \sum_{C} \ind(C) 
    \rtr(\rho(C)) \, \ell(C)^s \,,
  \end{equation}
  where the sum is taken over all free homotopical classes of closed
  orbits of the Reeb flow.
\end{thm}

We first make some comments about this identity. First, to remove the
real part in \eqref{eq:56}, one could also sum over all orbits $C$
\emph{and} $C^{-1}$, the latter corresponding to the opposite flow
$-T$. Indeed the torsion function is not sensitive to a change of
$\theta\mapsto -\theta$, together with $J\mapsto -J$, since $\kappa$
is defined using real operators $\Delta_0$ and $\Delta_1$.

Also let
\begin{equation}
  \label{eq:57}
  Z_\rho(s) = \sum_{C} \ind(C)
  \rtr(\rho(C))  \ell(C)^s
\end{equation}
be the dynamical zeta side of \eqref{eq:56}. From Proposition
\ref{prop:ind} this converges for $\Re (s) < 0$. In contrast the
spectral side
$$
\kappa^*(s/2) = \kappa (s/2) - \kappa (M, \rho)=
2\tr^*(\Delta_0^{-s/2}) - \tr^*(\Delta_1^{-s/2})
$$
is a converging series for $\Re (s) > 2$. Hence, when seen as
series, the spectral and dynamical sides of \eqref{eq:56} \emph{never}
converge for the same $s$, and the identity only holds through
meromorphic continuation.  When $s\to 0$ we get
\begin{equation}
  \label{eq:71}
  \lim_{s \rightarrow 0} \Bigl( Z_\rho(s) + \frac{ \kappa(M, \rho)}{s}
  \Bigr) = \kappa'(0)/2 = -
  \ln (T_{RS}(M, \rho))\,, 
\end{equation}
and ($\ln$ of) the analytic torsion may be seen as a topological
regularisation of the \emph{formal} dynamical series
\begin{equation}
  \label{eq:59}
  - ''Z_\rho(0)'' = -\sum_{C} \ind(C) \rtr (\rho(C)) \,. 
\end{equation} 
Comparing with \eqref{eq:55} and Fried's dynamical function yields
$$
''Z_\rho(0)'' = Z_\rho(0) = -\Re (Z_F(0))\,,
$$
in the acyclic case, and the dynamical functions $Z_F$ and $Z_\rho$
both provide analytic continuation of the same dynamical series in
\eqref{eq:59}.  This series has been interpreted in \cite{Fried87,
  Fried} as being the total Fuller measure of periodic orbits, and has
a formal invariance by deformation of the flow, as long as orbit
periods stay bounded.

\smallskip

We note also that the trace formula \eqref{eq:56} can be written in a
more symmetric manner
\begin{equation}
  \label{eq:60}
  \Gamma(s) \kappa^*(s) = \frac{2^{1-2s}}{
    \sqrt \pi} \Gamma(\frac{1}{2} - s) Z_\rho(2s)\,,
\end{equation}
as follows from the classical identities (see \cite{WW})
$$
\Gamma(s) \Gamma(s+ \frac{1}{2}) = 2^{1-2s} \sqrt{ \pi} \Gamma(2s)
\quad \mathrm{and} \quad \Gamma(s + \frac{1}{2}) \Gamma( -s +
\frac{1}{2}) = \frac{\pi}{\cos (\pi s)}\,.
$$
This formulation will be useful in \S\ref{sec:heat-kernel-as}.

\smallskip

As a last comment, we observe that the trace formula \eqref{eq:56} is
homogeneous in the constant rescaling $\theta\mapsto K \theta$.
Indeed, the metric here is $g = d\theta(\cdot, J \cdot) + \theta^2$,
hence $\ell(C)= \int_C \theta $ changes to $K \ell(C)$, while the
fourth-order contact Laplacians $\Delta_i $ are homogeneous and
rescale to $K^{-2} \Delta_i$. Thus $\zeta^*(\Delta^{1/2})(s) $
rescales to $K^s\zeta^*(\Delta^{1/2})(s)$ as needed. Such a property
does not hold for the Hodge--de Rham Laplacians.

\begin{proof}[Proof of Theorem \ref{thm:torsion-zeta}]
  We will start from the expression of $\kappa(s)$ by Hurwitz zeta
  functions as given in Theorem \ref{thm:kappa_x}. First Hurwitz's
  formula (see \cite[\S 13]{WW}) states that for $\Re(s)< 0$
  $$
  \zeta(s,x) = \frac{2\Gamma(1-s)}{(2\pi)^{1-s}} \Bigr[
  \sin(\frac{\pi s}{2}) \sum_{n=1}^{+\infty} \cos\bigl(\frac{2\pi
    xn}{n^{1-s}} \bigr) + \cos(\frac{\pi s}{2}) \sum_{n=1}^{+\infty}
  \sin \bigl(\frac{2\pi xn}{n^{1-s}} \bigr)\Bigr]\,,
  $$
  so that using $f(s) f(1-s) = \pi/2$ gives
  \begin{equation}
    \label{eq:61}
    f(s) (\zeta(s,x) + \zeta(s, 1-x))  = \sum_{n=1}^{+\infty} \Re
    \bigl(e^{2i\pi xn})\frac{(2\pi n)^s }{n}\,.
  \end{equation}
  Note also the corresponding limit expression when $x \rightarrow
  0^+$ with $\Re (s) <0$:
  \begin{equation}
    \label{eq:62}
    2f(s) \zeta(s) = \sum_{n=1}^{+\infty} \frac{(2\pi n)^s }{n}\,.
  \end{equation}
  
  $\bullet$ We study the contribution of $\kappa_x$ on $V^x$ with $x
  \in (0,1)$. By \eqref{eq:61}, formula \eqref{eq:49} yields
  \begin{displaymath}
    f(s) \kappa_x(s/2)= \chi(\Sigma^*) \sum_{n\geq 1} \frac{
      \rtr(\rho(f^n))}{n} (2\pi n)^s 
    + \sum_i\sum_{n\geq 1} \frac{\rtr(\rho(f_i^n))}{n}
    \bigl(\frac{2\pi n}{\alpha_i}\bigr)^s\,,
  \end{displaymath}
  hence by Proposition \ref{prop:ind},
  \begin{align*}
    f(s) \kappa_x(s/2) & = \sum_{n\geq 1} \bigl(\ind(f^n) - \sum_i
    \frac{1}{\alpha_i} \bigr)
    \rtr(\rho(f^n)) \ell(f^n)^s \\
    & \quad +\sum_i \sum_{n \notin \alpha_i \N} \ind(f_i^n)
    \rtr(\rho(f_i^n)) \ell(f_i^n)^s \\
    & \quad + \sum_i \sum_{k \geq 1} \frac{\rtr(\rho(f^k))}{k
      \alpha_i }\ell(f^k)^s \\
    & = \sum_C \ind(C) \rtr (\rho(C))\, l(C)^s \,,
  \end{align*}
  as needed in \eqref{eq:56}.
  
  \medskip
  
  $\bullet$ We now study $\kappa_0$ on $V^0 = \ker (\mathrm{Id} -
  \rho(f))$. Recall that $\spec \rho(f_i) = \{e^{2i \pi x_{i,j}}\}$
  and $V^{0,i} = \ker (\mathrm{Id} - \rho(f_i))$. By
  \eqref{eq:61}--\eqref{eq:62} formula \eqref{eq:50} reads
  \begin{multline}
    \label{eq:63}
    f(s) \bigl( \kappa_0(s/2) - \kappa(M, \rho) \bigr) =
    \kappa(M,\rho) \sum_{n\geq 1} \frac{(2\pi n)^s }{n} + \sum_i
    \sum_{n\geq 1} \frac{\rtr_{(V^{0,i})^\bot}(\rho(f_i^n))}{n}
    \bigl(\frac{2\pi
      n}{\alpha_i}\bigr)^s  \\
    + \sum_i \sum_{n\geq 1} \dim(V^{0,i}) (\alpha_i^{-s} -1)
    \frac{(2\pi n)^s}{n}\,.
  \end{multline}
  By \eqref{eq:52}, the first series reads
  \begin{align*}
    \kappa(M,\rho) \sum_{n\geq 1} \frac{(2\pi n)^s }{n} & =
    \bigl[\dim(V^0) (\chi(\Sigma) - \sum_i \frac{1}{\alpha_i}) +
    \sum_i \dim (V_0^i)\bigr] \sum_{n\geq 1} \frac{(2\pi
      n)^s }{n} \\
    & = \sum_{n\geq 1} \ind(f^n) \rtr (\rho(f^n)) \ell(f^n)^s - \sum_i
    \sum_{n\geq 1} \dim (V^0) \frac{(2\pi n)^s}{n \alpha_i} \\
    & \quad + \sum_i \sum_{n\geq 1} \dim (V^{0,i}) \frac{(2\pi n)^s
    }{n } \,.
  \end{align*}
  Since $\rho(f_i) = \mathrm{Id}$ on $V^{0,i}$, the second series in
  \eqref{eq:63} splits into
  \begin{align*}
    \sum_i \sum_{n\geq 1} \frac{\rtr_{(V^{0,i})^\bot}(\rho(f_i^n))}{n}
    \bigl(\frac{2\pi n}{\alpha_i}\bigr)^s &= \sum_i \sum_{n\geq 1}
    \frac{\rtr(\rho(f_i^n))}{n} \bigl(\frac{2\pi n}{\alpha_i}\bigr)^s
    - \sum_i \sum_{n\geq 1} \frac{\dim V^{0,i}}{n} \bigl(\frac{2\pi
      n}{\alpha_i}\bigr)^s
    \\
    & = \sum_i \sum_{n \notin
      \alpha_i \N} \ind(f_i^n) \rtr(\rho(f_i^n) ) \, \ell(f_i^n)^s \\
    & \quad + \sum_i \sum_{k\geq 1}\dim V^0 \frac{(2\pi k)^s}{k
      \alpha_i} - \sum_i \sum_{n\geq 1} \frac{\dim V^{0,i}}{n
      \alpha_i^s} (2\pi n)^s
  \end{align*}
  since $\rho(f_i^n) = \rho(f^k) = \mathrm{Id}$ on $V^0$ for $n = k
  \alpha_i$. Therefore after cancellations \eqref{eq:63} yields
  \begin{displaymath}
    f(s) (\kappa_0(s/2) - \kappa(M, \rho)) = \sum_C \ind (C)
    \rtr(\rho(C)) \ell(C)^s\,,
  \end{displaymath}
  as needed.
\end{proof}

\subsection{The torsion heat trace as a dynamical theta function}
\label{sec:heat-kernel-as}

In Theorem \ref{thm:torsion-zeta}, spectral and dynamical aspects of
analytic torsion are compared through zeta functions. One can also
work at the level of heat kernels.  Consider the heat operators of the
fourth-order Laplacians $\Delta_0$ and $\Delta_1$ of the contact
complex, and set
\begin{align}
  \label{eq:64}
  \tr_\kappa (e^{-t \Delta}) & = 2 \tr (e^{-t \Delta_0}) - \tr
  (e^{-t \Delta_1}) \,\\
  & = \tr_\kappa^* (e^{-t \Delta}) + \kappa(M, \rho) \nonumber\,.
\end{align}
Recall that for $\Re (s) >1$,
\begin{equation}
  \label{eq:65}
  \kappa^*(s) = 2 \zeta^*(\Delta_0)(s) - \zeta^*(\Delta_1)(s) =
  \frac{1}{\Gamma(s)} \int_0^{+\infty} t^{s-1} \tr_\kappa^*(e^{-t
    \Delta}) dt \,.
\end{equation}
Then the following trace formula holds in our CR Seifert setting.
\begin{thm}
  \label{thm:Selberg}
  One has
  \begin{equation}
    \label{eq:66}
    \tr_\kappa(e^{-t \Delta}) = \dim V \frac{\sqrt \pi
      \chi(\Sigma)}{\sqrt t } + 
    \frac{1}{\sqrt{\pi t} }\sum_C 
    \ell(C) \ind(C) \rtr(\rho(C)) e^{-\ell(C)^2/ 4t}\,, 
  \end{equation}
  where $C$ runs over free homotopical classes of closed orbits of the
  Reeb flow and $\chi(\Sigma) $ is the rational Euler class of the
  quotient orbifold.
\end{thm}
Hence the torsion heat trace may be compared with the purely dynamical
theta function
\begin{equation}
  \label{eq:67}
  \vartheta (t) = \frac{1}{\sqrt{\pi t} }\sum_C 
  \ell(C) \ind(C) \rtr(\rho(C)) e^{-\ell(C)^2/ 4t}\,.
\end{equation}

We will first need the following fact on the asymptotic heat
development.
\begin{prop}
  \label{prop:torsion-heat}
  As $t\searrow 0$, it holds that
  $$
  \tr_\kappa (e^{-t \Delta}) = \dim V \frac{\sqrt \pi \chi
    (\Sigma)}{\sqrt t} + O(\sqrt t )\,.
  $$
\end{prop}
\begin{proof}
  By Corollary \ref{cor:residue_kappa}, the torsion function $\kappa $
  has a single simple pole at $s=1/2$ with residue $\chi(\Sigma) \dim
  V$. On the other hand we know that, for a fourth-order hypoelliptic
  Laplacian in dimension $3$, as $t \searrow 0$,
  $$
  \tr_\kappa (e^{-t \Delta}) = \frac{c_1}{t} + \frac{c_{1/2}}{\sqrt
    t} + c_0 + O(\sqrt t)\,.
  $$
  Mellin's transform \eqref{eq:65} split into $\int_0^1 +
  \int_1^{+\infty}$ providing
  $$
  c_1= \mathrm{Res}_{s=1} (\kappa(s)) = 0\,,\quad c_{1/2} =
  \Gamma(1/2) \mathrm{Res}_{1/2} (\kappa) = \sqrt\pi \chi(\Sigma) \dim
  V \,,
  $$
  and
  $$
  c_0 = \kappa(M, \rho) + \mathrm{Res}_0 (\Gamma(s) \kappa^*(s)) =
  0\,,
  $$
  since $\kappa^*(0)= \kappa(0) - \kappa(M,\rho) = - \kappa(M,
  \rho)$.
\end{proof}

\smallskip

We can now prove Theorem \ref{thm:Selberg}.

\begin{proof}
  One takes Mellin transforms $\mathcal{M}$ of both sides in
  \eqref{eq:66}.  For $\Re (s) < 0$ one finds
  \begin{equation}
    \label{eq:68}
    \mathcal{M}( \vartheta)(s)   = \frac{2^{1-2s}}{
      \sqrt \pi} \Gamma(\frac{1}{2} - s) Z_\rho(2s)\,.
  \end{equation}
  On the other hand \eqref{eq:65} and Proposition
  \ref{prop:torsion-heat} yield that, for $\Re (s) > 1/2$,
  $$
  \mathcal{M}(\tr_\kappa^*(e^{-t \Delta}))(s) = \Gamma(s)
  \kappa^*(s)\,.
  $$
  In order to compare these identities we need first to extend them
  to a common domain. Indeed set
  $$\tr_0 (e^{-t \Delta}) = \tr^*_\kappa(e^{-t \Delta}) +
  \Bigl(\kappa(M, \rho) - \dim V \frac{\sqrt \pi \chi (\Sigma)}{\sqrt
    t}\Bigr)\chi_{]0,1]}(t);
  $$
  by Proposition \ref{prop:torsion-heat} this has a Mellin
  transform for $\Re (s) > -1/2$ and
  \begin{equation}
    \mathcal{M}(\tr_0 (e^{-t \Delta}))(s) = \Gamma(s) \kappa^*(s) +
    \frac{\kappa(M, \rho)}{s} + \dim V \frac{\sqrt \pi \chi
      (\Sigma)}{s-1/2}\,. 
  \end{equation}
  Set also
  \begin{equation*}
    \vartheta_0(t) = \vartheta(t) - \Bigl( \kappa(M, \rho) - \dim V
    \frac{\sqrt \pi \chi(\Sigma)}{\sqrt t}\Bigr)\chi_{[1, +\infty[}(t) \,,
  \end{equation*}
  so that \eqref{eq:68} yields, for $\Re (s) < 0$,
  \begin{align*}
    \mathcal{M}(\vartheta_0)(s) & = \frac{2^{1-2s}}{ \sqrt \pi}
    \Gamma(\frac{1}{2} - s) Z_\rho(2s) + \frac{\kappa(M, \rho)}{s} +
    \dim V \frac{\sqrt \pi \chi
      (\Sigma)}{s-1/2} \\
    & = \mathcal{M}(\tr_0 (e^{-t \Delta}))(s)\,,
  \end{align*}
  for $-1/2 < \Re (s) < 0$, by \eqref{eq:60} and \eqref{eq:68}. Hence
  by injectivity of the Mellin transform, coming from Fourier
  injectivity on integrable functions here, one concludes that
  $\vartheta_0 (t)= \tr_0 (e^{-t \Delta})$, yielding the trace
  formula.
  
  The authors thank Patrick G{\'e}rard for an enlightening discussion
  on this proof.
\end{proof}

Formula \eqref{eq:66} is a typical Selberg-type trace formula, which
holds in many other geometric situations, see \cite[p.~57]{Fried87}
for instance.  Note that it holds here even in \emph{variable}
curvature.  This may look unusual as Selberg's technique is algebraic
and relies on group actions, hence makes sense on uniformised (locally
symmetric) manifolds. However we know, by \eqref{eq:40} or Theorem
\ref{thm:Lefschetz-torsion}, that on CR Seifert manifolds the torsion
function $\kappa(s)$ is `topological', meaning independent of the
complex structure $J$ since $\theta$ is fixed by the circle action.
Now except for some cases that fibre over the sphere $S^2$ with two
singular points, all other CR Seifert manifolds can be uniformised,
i.e.~endowed with a constant curvature metric; see e.g.~\cite{Belgun}
or \cite[Theorems 1.1, 1.2]{FS}.  By Moser's lemma this can be done
with a fixed volume.  Thus, except for the special cases mentioned,
one could have worked over a uniformised orbifold $\Sigma$, where
Selberg's technique should also be applicable.

\medskip

The trace formula \eqref{eq:66} has a striking consequence for the
small time behaviour of the torsion heat trace $\tr_\kappa( e^{-t
  \Delta})$.  Indeed, from its definition \eqref{eq:67} and
Proposition \ref{prop:ind}, the dynamical theta function $\vartheta$
clearly decays very fast as
$$
\vartheta(t) = O(e^{-C /t}) \quad \mathrm{when} \quad t \searrow
0\,,
$$
so that instead of Proposition \ref{prop:torsion-heat} we get the
\emph{full} heat development.
\begin{cor}
  \label{cor:full-heat}
  On CR Seifert manifolds, as $t\searrow 0$ we have
  $$
  \tr_\kappa (e^{-t \Delta}) = \dim V \frac{\sqrt \pi \chi
    (\Sigma)}{\sqrt t} + O(e^{-C/t})\,.
  $$
\end{cor}
Thus on such manifolds this torsion heat trace does show a `fantastic
cancellations' phenomenon as encountered for the heat of Dirac
operators.

Note that the only surviving term as $t\searrow 0$ in this heat
development is the integral of curvature data, as should be the case.
Gauss--Bonnet reads here
$$
\int_M R \, \theta \wedge d \theta = 2 \pi \int_\Sigma R d \theta =
4\pi^2 \chi(\Sigma)\,,
$$
where $R$ stands for the Tanaka--Webster scalar curvature, which
coincides with the Riemannian scalar curvature of the base in this
Seifert case; see e.g.~\cite{BHR}.  Then, by universality of the
coefficients of the heat development, the following `fantastic
factorisation' holds on \emph{any} $3$-dimensional contact manifold.
\begin{cor}
  \label{cor:full-heat-contact} On any contact $3$-manifold, the full
  development of $\tr_\kappa (e^{-t \Delta})$ as $t\searrow 0$ is of
  type
  \begin{equation}
    \label{eq:69}
    \tr_\kappa (e^{-t \Delta}) \sim \frac{\dim V}{4 \pi \sqrt{\pi
        t}} \int_M R \,\theta \wedge d 
    \theta   + \sum_{n\geq 0} t^{n/2} \int_M P_n(R,A) d
    \mathrm{vol} \,,   
  \end{equation}
  where all invariant curvature polynomials $P_n(R,A)$ factorise
  through Tanaka--Webster torsion $A = \mathcal{L}_T J$ and its
  covariant derivatives.
\end{cor}

\smallskip

In the opposite direction, when $t\to +\infty$, the trace formula
\eqref{eq:66} gives the asymptotic development of the dynamical theta
function $\vartheta$. Indeed, after removing the zero eigenspace, heat
decays exponentially, yielding the following property.
\begin{cor}
  \label{cor:theta-decay}
  On CR Seifert manifolds, it holds as $t \rightarrow +\infty$ that
  \begin{equation}
    \label{eq:70}
    \vartheta(t) = \kappa(M, \rho) - \dim V \frac{\sqrt \pi
      \chi(\Sigma)}{\sqrt t} + O(e^{-Ct})\,.
  \end{equation}
\end{cor}
This behaviour is not surprising; it can also be seen from the
explicit formula \eqref{eq:67} and Proposition \ref{prop:ind}, which
relate the dynamical theta function to the classical Jacobi theta
function, whose decay at $+\infty$ is well known; see e.g.~\cite[\S
21.51]{WW}.

\bibliographystyle{abbrv}


\begin{thebibliography}{10}

\bibitem{BM}
C.~B{\"a}r and S.~Moroianu.
\newblock Heat kernel asymptotics for roots of generalized {L}aplacians.
\newblock {\em Internat. J. Math.}, 14(4):397--412, 2003.

\bibitem{BE}
R.~J. Baston and M.~G. Eastwood.
\newblock {\em The {P}enrose transform: Its interaction with representation
  theory}.
\newblock Oxford Mathematical Monographs. The Clarendon Press, Oxford
  University Press, New York, 1989.

\bibitem{BG}
R.~Beals and P.~Greiner.
\newblock {\em Calculus on {H}eisenberg manifolds}, volume 119 of {\em Annals
  of Mathematics Studies}.
\newblock Princeton University Press, Princeton, NJ, 1988.

\bibitem{BGS}
R.~Beals, P.~C. Greiner, and N.~K. Stanton.
\newblock The heat equation and geometry of {CR} manifolds.
\newblock {\em Bull. Amer. Math. Soc. (N.S.)}, 10(2):275--276, 1984.

\bibitem{Belgun}
F.~A. Belgun.
\newblock Normal {CR} structures on {$S\sp 3$}.
\newblock {\em Math. Z.}, 244(1):125--151, 2003.

\bibitem{BH}
O.~Biquard and M.~Herzlich.
\newblock A {B}urns-{E}pstein invariant for {ACHE} 4-manifolds.
\newblock {\em Duke Math. J.}, 126(1):53--100, 2005.

\bibitem{BHR}
O.~Biquard, M.~Herzlich, and M.~Rumin.
\newblock Diabatic limit, eta invariants and {C}auchy-{R}iemann manifolds of
  dimension $3$.
\newblock {\em Ann. Sci. Ecole Norm. Sup. (4)}, 40(4):589--631, 2007.

\bibitem{Bismut05}
J.-M. Bismut.
\newblock The hypoelliptic {L}aplacian on the cotangent bundle.
\newblock {\em J. Amer. Math. Soc.}, 18(2):379--476 (electronic), 2005.

\bibitem{Bismut07}
J.-M. Bismut.
\newblock Loop spaces and the hypoelliptic {L}aplacian.
\newblock To appear in Comm. Pure Appl. Math., 2008.

\bibitem{BGSI}
J.-M. Bismut, H.~Gillet, and C.~Soul{\'e}.
\newblock Analytic torsion and holomorphic determinant bundles. {I}.
  {B}ott-{C}hern forms and analytic torsion.
\newblock {\em Comm. Math. Phys.}, 115(1):49--78, 1988.

\bibitem{BGSIII}
J.-M. Bismut, H.~Gillet, and C.~Soul{\'e}.
\newblock Analytic torsion and holomorphic determinant bundles. {III}.
  {Q}uillen metrics on holomorphic determinants.
\newblock {\em Comm. Math. Phys.}, 115(2):301--351, 1988.

\bibitem{BL}
J.-M. Bismut and G.~Lebeau.
\newblock The hypoelliptic {L}aplacian and {R}ay--{S}inger metrics.
\newblock To appear, 2008.

\bibitem{BZ}
J.-M. Bismut and W.~Zhang.
\newblock An extension of a theorem by {C}heeger and {M}\"uller. {W}ith an
  appendix by {F}ran\c cois {L}audenbach.
\newblock {\em Ast\'erisque}, (205):235, 1992.

\bibitem{Branson}
T.~Branson.
\newblock {$Q$}-curvature and spectral invariants.
\newblock {\em Rend. Circ. Mat. Palermo (2) Suppl.}, (75):11--55, 2005.

\bibitem{Cheeger}
J.~Cheeger.
\newblock Analytic torsion and the heat equation.
\newblock {\em Ann. of Math. (2)}, 109(2):259--322, 1979.

\bibitem{Fried87}
D.~Fried.
\newblock Lefschetz formulas for flows.
\newblock In {\em The Lefschetz centennial conference, Part III (Mexico City,
  1984)}, volume~58 of {\em Contemp. Math.}, pages 19--69. Amer. Math. Soc.,
  Providence, RI, 1987.

\bibitem{Fried}
D.~Fried.
\newblock Counting circles.
\newblock In {\em Dynamical systems (College Park, MD, 1986--87)}, volume 1342
  of {\em Lecture Notes in Math.}, pages 196--215. Springer, Berlin, 1988.

\bibitem{Fried88}
D.~Fried.
\newblock Torsion and closed geodesics on complex hyperbolic manifolds.
\newblock {\em Invent. Math.}, 91(1):31--51, 1988.

\bibitem{Fuller}
F.~B. Fuller.
\newblock An index of fixed point type for periodic orbits.
\newblock {\em Amer. J. Math.}, 89:133--148, 1967.

\bibitem{FS}
M.~Furuta and B.~Steer.
\newblock Seifert fibred homology {$3$}-spheres and the {Y}ang-{M}ills
  equations on {R}iemann surfaces with marked points.
\newblock {\em Adv. Math.}, 96(1):38--102, 1992.

\bibitem{Getzler}
E.~Getzler.
\newblock An analogue of {D}emailly's inequality for strictly pseudoconvex {CR}
  manifolds.
\newblock {\em J. Differential Geom.}, 29(2):231--244, 1989.

\bibitem{Gilkey}
P.~B. Gilkey.
\newblock {\em Invariance theory, the heat equation, and the {A}tiyah-{S}inger
  index theorem}, volume~11 of {\em Mathematics Lecture Series}.
\newblock Publish or Perish Inc., Wilmington, DE, 1984.

\bibitem{JK}
P.~Julg and G.~Kasparov.
\newblock Operator {$K$}-theory for the group {${\rm SU}(n,1)$}.
\newblock {\em J. Reine Angew. Math.}, 463:99--152, 1995.

\bibitem{Kassel}
C.~Kassel.
\newblock Le r\'esidu non commutatif (d'apr\`es {M}.\ {W}odzicki).
\newblock {\em Ast\'erisque}, (177-178):Exp.\ No.\ 708, 199--229, 1989.
\newblock S\'eminaire Bourbaki, Vol.\ 1988/89.

\bibitem{KM}
F.~F. Knudsen and D.~Mumford.
\newblock The projectivity of the moduli space of stable curves. {I}.
  {P}reliminaries on ``det'' and ``{D}iv''.
\newblock {\em Math. Scand.}, 39(1):19--55, 1976.

\bibitem{Milnor}
J.~W. Milnor and J.~D. Stasheff.
\newblock {\em Characteristic classes}, volume~76 of {\em Annals of Mathematics
  Studies}.
\newblock Princeton University Press, Princeton, NJ, 1974.

\bibitem{MS}
H.~Moscovici and R.~J. Stanton.
\newblock {$R$}-torsion and zeta functions for locally symmetric manifolds.
\newblock {\em Invent. Math.}, 105(1):185--216, 1991.

\bibitem{Muller}
W.~M{\"u}ller.
\newblock Analytic torsion and {$R$}-torsion of {R}iemannian manifolds.
\newblock {\em Adv. Math.}, 28(3):233--305, 1978.

\bibitem{Nicolaescu}
L.~I. Nicolaescu.
\newblock Finite energy {S}eiberg-{W}itten moduli spaces on 4-manifolds
  bounding {S}eifert fibrations.
\newblock {\em Comm. Anal. Geom.}, 8(5):1027--1096, 2000.

\bibitem{PongeAMS}
R.~Ponge.
\newblock Heisenberg calculus and the spectral theory of hypoelliptic operators
  on {H}eisenberg manifolds.
\newblock Preprint. arxiv:math.AP/0509300. To appear in Mem. Amer. Math. Soc.

\bibitem{PongeJFA}
R.~Ponge.
\newblock Noncommutative residue for {H}eisenberg manifolds. {A}pplications in
  {CR} and contact geometry.
\newblock {\em J. Funct. Anal.}, 252:399--463, 2007.

\bibitem{Quillen}
D.~Quillen.
\newblock Determinants of {C}auchy-{R}iemann operators on {R}iemann surfaces.
\newblock {\em Funct. Anal. Appl.}, 14:31--34, 1985.

\bibitem{RS}
D.~B. Ray and I.~M. Singer.
\newblock {$R$}-torsion and the {L}aplacian on {R}iemannian manifolds.
\newblock {\em Adv. Math.}, 7:145--210, 1971.

\bibitem{Rockland}
C.~Rockland.
\newblock Hypoellipticity on the {H}eisenberg group--representation-theoretic
  criteria.
\newblock {\em Trans. Amer. Math. Soc.}, 240:1--52, 1978.

\bibitem{rosenberg}
S.~Rosenberg.
\newblock {\em The {L}aplacian on a {R}iemannian manifold. {A}n introduction to
  analysis on manifolds}, volume~31 of {\em London Mathematical Society Student
  Texts}.
\newblock Cambridge University Press, Cambridge, 1997.

\bibitem{Rumin90}
M.~Rumin.
\newblock Un complexe de formes diff\'erentielles sur les vari\'et\'es de
  contact.
\newblock {\em C. R. Acad. Sci. Paris S\'er. I Math.}, 310(6):401--404, 1990.

\bibitem{Rumin94}
M.~Rumin.
\newblock Formes diff\'erentielles sur les vari\'et\'es de contact.
\newblock {\em J. Differential Geom.}, 39(2):281--330, 1994.

\bibitem{Rumin00}
M.~Rumin.
\newblock Sub-{R}iemannian limit of the differential form spectrum of contact
  manifolds.
\newblock {\em Geom. Funct. Anal.}, 10(2):407--452, 2000.

\bibitem{Scott}
P.~Scott.
\newblock The geometries of {$3$}-manifolds.
\newblock {\em Bull. London Math. Soc.}, 15(5):401--487, 1983.

\bibitem{Seshadri-ACH}
N.~Seshadri.
\newblock Approximately {E}instein {ACH} metrics, volume renormalization, and
  an invariant for contact manifolds.
\newblock Preprint. arXiv:0707.587.

\bibitem{Stanton}
N.~K. Stanton.
\newblock Spectral invariants of {CR} manifolds.
\newblock {\em Michigan Math. J.}, 36(2):267--288, 1989.

\bibitem{Tanaka}
N.~Tanaka.
\newblock {\em A differential geometric study on strongly pseudo-convex
  manifolds}.
\newblock Lectures in Mathematics, Department of Mathematics, Kyoto University,
  No. 9. Kinokuniya Book-Store Co. Ltd., Tokyo, 1975.

\bibitem{Tanno}
S.~Tanno.
\newblock Variational problems on contact {R}iemannian manifolds.
\newblock {\em Trans. Amer. Math. Soc.}, 314(1):349--379, 1989.

\bibitem{Taylor}
M.~E. Taylor.
\newblock Noncommutative microlocal analysis. {I}.
\newblock {\em Mem. Amer. Math. Soc.}, 52(313):iv+182, 1984.

\bibitem{Webster}
S.~M. Webster.
\newblock Pseudo-{H}ermitian structures on a real hypersurface.
\newblock {\em J. Differential Geom.}, 13(1):25--41, 1978.

\bibitem{WW}
E.~T. Whittaker and G.~N. Watson.
\newblock {\em A course of modern analysis. {A}n introduction to the general
  theory of infinite processes and of analytic functions: with an account of
  the principal transcendental functions}.
\newblock Fourth edition. Reprinted. Cambridge University Press, New York,
  1962.

\end{thebibliography}



-----------------------------------------------------

\end{document}